\documentclass[10pt,a4paper,oneside]{article}

\usepackage{algorithm}
\usepackage{algpseudocode}
\usepackage{bigints}
\usepackage{graphicx}
\usepackage{epstopdf}
\usepackage[noblocks]{authblk}
\usepackage{amssymb}
\usepackage{amsmath}
\usepackage{amsfonts}
\usepackage{mathtools}
\usepackage{color}
\usepackage{listings}
\usepackage{fancyvrb}
\usepackage{float}
\usepackage{framed}
\usepackage{multirow}
\usepackage{multicol}
\usepackage{subcaption}
\usepackage{pifont}
\usepackage{colortbl}
\usepackage{framed}
\usepackage{listings}
\usepackage{bm}
\usepackage{eurosym}
\usepackage{amssymb}
\usepackage{suffix}
\usepackage{tikz}
\usepackage{bbm}
\usetikzlibrary{matrix}
\usepackage[customcolors]{hf-tikz}
\usepackage{booktabs}

\usepackage[toc,page]{appendix}

\usepackage[a4paper,left=2cm,right=2cm,top=2.5cm,bottom=2.5cm]{geometry}
\usepackage[font=small]{caption}
\usepackage[hidelinks,unicode]{hyperref}

\newtheorem{theorem}{Theorem}[section]
\newtheorem{proposition}[theorem]{Proposition}

\newtheorem{lemma}[theorem]{Lemma}
\newtheorem{remark}[theorem]{Remark}

\newenvironment{proof}{\noindent{\bf Proof:}}{\hfill$\square$}

\DeclareMathOperator{\arcsinh}{arcsinh}
\DeclareMathOperator{\erfc}{erfc}

\numberwithin{equation}{section}

\title{On the computation of the cumulative distribution function of the Normal Inverse Gaussian distribution}
\author{
 \normalsize  Guillermo Navas-Palencia\\
  \texttt{\normalsize  g.navas.palencia@gmail.com}
}

\date{\small \today}

\begin{document}
\maketitle

\abstract{In this paper, we obtain various series and asymptotic expansions involving the modified Bessel function of the second kind for the normal inverse Gaussian cumulative distribution function. The new expansions accelerate computations, complementing the numerical integration methods implemented in statistical software packages. We also provide a detailed description of the algorithm and its corresponding implementation in C++. The performance and accuracy of the algorithm are extensively tested and benchmarked with open-source implementations, offering superior accuracy and speed-ups of a factor from 5 to 60.}

\section{Introduction}\label{sec1}

The normal inverse Gaussian (NIG) distribution is a flexible four-parameter distribution introduced by Barndorff-Nielsen in 1977 \cite{Barndorff1977}, which arises as a particular case of the generalized hyperbolic (GH) distribution. In the last two decades, it has gained significant popularity in various fields due to its desirable properties and the ability to model a wide range of shapes,  particularly suited for skewed and heavy-tailed data. 

The NIG distribution was initially introduced in the context of financial data modelling, with applications in option pricing to capture non-Gaussian asset returns \cite{Eriksson2009}, portfolio risk management and credit derivatives for pricing collateralized default obligations \cite{Prause1999, Kalemanova2007, Navas-Palencia_thesis2016}, stochastic modelling of energy derivatives \cite{Benth2004} and stochastic volatility modelling \cite{Barndorff1997}. Besides the numerous applications in financial modelling, it has found applications in wireless communication and remote sensing \cite{Hanssen2001, Solbo2004}, turbulence and fluid dynamics \cite{Hedevang2013}, astrophysics \cite{Louarn2024} and reliability analysis \cite{Xu2022}.

In this work, we focus on deriving series and asymptotic expansions to efficiently compute the cumulative distribution function of the normal inverse Gaussian. The new expansions, derived from the Laplace-type integral representation, are proposed to complement the direct use of numerical integration methods, the latter being the method of choice in all the implementations publicly available. 

The outline of the paper is as follows. In Section 2, we revisit the distribution properties of the NIG distribution and introduce alternative integral representations for the cumulative distribution function. Then, in Section 3, we derive and analyse a variety of series expansions and asymptotic expansions for two special cases and the general case, calculating, in many cases, rigorous error bounds. In addition, we provided a detailed treatment of the strategies to compute two possible integral representations using different quadrature methods. Section 4 presents an extensive discussion on algorithmic aspects to ease its implementation to the practitioners. After that, in Section 5, we assess the performance and accuracy of the proposed algorithms. Finally, Section 6 presents our conclusions and several possible enhancements.

\section{Distribution properties}

The NIG distribution is defined as a normal variance-mean mixture distribution, where the mixing distribution for the variance follows an inverse Gaussian (IG) distribution \cite[\S 9]{Paolella2007}. Consider $Z$ is a positive random variable following a distribution, $Z \sim \mathcal{IG}(\delta^2, \gamma^2)$ and $\mu$ and $\beta$ are constants. If $X \sim \mathcal{NIG}(\alpha, \beta, \mu, \delta)$, then
\begin{equation}
X \,|\, Z \sim \mathcal{N}(\mu + \beta Z, Z),
\end{equation}
where $\gamma = \sqrt{\alpha^2 - \beta^2}$.
The NIG distribution parameters\footnote{We use the parameterization dominating the literature. We refer to \cite{Paolella2007} and \cite{Prause1999} for an overview of different parameterizations.} admit the following interpretation: $\alpha$ is the tail heaviness, $\beta$ is the asymmetry or skewness, $\mu$ is the location and $\delta$ the scaling. If $\beta = 0$, the distribution is symmetric and centred around $\mu$. The domain of the parameters is
\begin{equation}
0 \le |\beta| < \alpha, \quad \mu \in \mathbb{R}, \quad \delta  > 0.
\end{equation}
In addition, a related parameter typically used to simplify notation is $\omega = \sqrt{\delta^2 + (x-\mu)^2}$.
Since the NIG distribution is a special case of the GH distribution, distributional properties, such as being closed under convolution and infinite divisibility, can also be inferred from results for the GH distribution \cite{Paolella2007}.

\subsection{Density function}
The NIG distribution has probability density function
\begin{align}
f(x; \alpha, \beta, \mu, \delta) &= \frac{\alpha \delta}{\pi} \frac{K_1\left(\alpha\sqrt{\delta^2 + (x-\mu)^2}\right)}{\sqrt{\delta^2 + (x-\mu)^2}} e^{\delta \sqrt{\alpha^2 - \beta^2} + \beta(x-\mu)}\nonumber\\
&= \frac{\alpha \delta}{\pi} \frac{K_1\left(\alpha\omega\right)}{\omega} e^{\delta \gamma + \beta(x-\mu)},
\end{align}
where $K_1(x)$ is the modified Bessel function of the second kind \cite[\S 10]{NIST:DLMF},  hereafter called modified Bessel function. Note that this function is ubiquitous throughout this work. As an additional note, the case $\mu = 0$ and $\delta = 1$ is often defined as the standardized form.

\subsection{Cumulative distribution function}\label{properties_cdf}
The cumulative distribution function has the common integral representation from its density function,
\begin{equation}\label{integral_k1}
F(x; \alpha, \beta, \mu, \delta) = \frac{\alpha \delta e^{\delta \gamma}}{\pi} \int_{-\infty}^{x} \frac{K_1\left(\alpha\sqrt{\delta^2 + (t-\mu)^2}\right)}{\sqrt{\delta^2 + (t-\mu)^2}} e^{\beta(t-\mu)} \mathop{dt}.
\end{equation}

On the other hand, using the variance-mean mixture distribution properties \cite[\S 9]{Paolella2007}, we have
\begin{equation}\label{integral_phi}
F(x; \alpha, \beta, \mu, \delta) = \frac{\delta}{\sqrt{2\pi}}\int_{0}^{\infty} \Phi\left(\frac{x - (\mu +\beta t)}{\sqrt{t}}\right) t^{-3/2} e^{-\frac{(\delta - \gamma t)^2}{2t}} \mathop{dt},
\end{equation}
where $\Phi(x)$ defined as
\begin{equation*}
\Phi(x) = \frac{1}{\sqrt{2\pi}}\int_{-\infty}^x e^{-t^2 / 2} \mathop{dt}
\end{equation*}
is the standard normal cumulative distribution function. From the integral in \eqref{integral_phi}, it follows from the reflection formula of $\Phi(x) = 1 - \Phi(-x)$ that
\begin{equation}\label{cdf_mirror}
F(x; \alpha, \beta, \mu, \delta) = 1- F(-x; \alpha, -\beta, -\mu, \delta).
\end{equation}
The relation becomes indispensable to extend the domain of applicability of multiple expansions derived in the next section. Furthermore, it can also help to minimize numerical issues when computing the complementarity cumulative distribution function or survival function. Finally, authors in \cite{Kalemanova2007} propose applying a logarithmic transformation to transform the integration interval to $[0, 1]$, arguing better numerics.

\subsubsection{Alternative integral representations}
Following, we obtain alternative integral representations in terms of other special functions. Particularly interesting is the derivation of an integral representation in terms of elementary functions. Although these representations have restrictions on the sign of $x-\mu$, the use of the reflection formula \eqref{cdf_mirror} permits the computation on the whole domain.

\begin{proposition} An incomplete Laplace-type integral representation in terms of modified Bessel function $K_0(x)$ for $x - \mu < 0$ is given by
\begin{equation}\label{integral_k0}
F(x; \alpha, \beta, \mu, \delta) = \frac{\sqrt{2}\delta e^{\delta \gamma}}{\pi}\int_{\beta / \sqrt{2}}^{\infty} e^{\sqrt{2}(x-\mu)t} K_0\left(\sqrt{2\left( (x-\mu)^2 + \delta^2\right)} \sqrt{\frac{\gamma^2}{2} + t^2}\right) \mathop{dt}.
\end{equation}
\end{proposition}
\begin{proof}
Consider the integral representation of the function $\Phi\left(\frac{x-\mu}{\sqrt{t}} -\beta \sqrt{t} \right)$ in \cite[\S 7.7]{NIST:DLMF}
\begin{equation*}
\Phi\left(\frac{x-\mu}{\sqrt{t}} -\beta \sqrt{t} \right) = \sqrt{\frac{t}{\pi}} e^{-\frac{(x-\mu)^2}{2t}}\int_{\beta/\sqrt{2}}^{\infty} e^{-(tu^2 - \sqrt{2}(x-\mu) u)} \mathop{du}.
\end{equation*}
Replacing in \eqref{integral_phi} and interchanging the order of integration, we obtain
\begin{equation*}
F(x; \alpha, \beta, \mu, \delta) = \frac{\delta e^{\delta \gamma}}{\sqrt{2}\pi} \int_{\beta/\sqrt{2}}^{\infty} e^{\sqrt{2}(x-\mu) u} \int_{0}^{\infty} t^{-1} e^{-\frac{\left((x-\mu)^2 + \delta^2\right)}{2t} - \left(\frac{\gamma^2}{2} + u^2\right)t} \mathop{dt} \mathop{du}.
\end{equation*}
Observing that the inner integral can be represented in terms of the modified Bessel function $K_0(x)$, we get
\begin{equation*}
\int_{0}^{\infty} t^{-1} e^{-\frac{\left((x-\mu)^2 + \delta^2\right)}{2t} - \left(\frac{\gamma^2}{2} + u^2\right)t} \mathop{dt} = 2 K_0\left(2\sqrt{\frac{(x-\mu)^2 + \delta^2}{2}} \sqrt{\frac{\gamma^2}{2} + t^2}\right).
\end{equation*}
\end{proof}

\begin{proposition}
An integral representation of the cumulative distribution function in terms of elementary functions is given by
\begin{equation}\label{integral_sine_transform}
F(x; \alpha, \beta, \mu, \delta) = 1 - \frac{e^{\delta \gamma}}{\pi}\int_0^{\infty} \frac{t e^{-(x-\mu)\left(\sqrt{t^2 + \alpha^2} - \beta\right)}}{\sqrt{t^2 + \alpha^2}\left(\sqrt{t^2 + \alpha^2} - \beta\right)}\sin(\delta t)\mathop{dt}, \quad x-\mu > 0.
\end{equation}
\end{proposition}
\begin{proof}
Consider the integral representation \eqref{integral_k1}. Then, we replace the term involving $K_1(x)$ by its sine transform \cite[\S 2.4]{Bateman1954}
\begin{equation*}
\frac{K_1(\alpha \sqrt{t^2 + \delta^2})}{\sqrt{t^2 + \delta^2}} = \frac{1}{\alpha\delta}\int_0^{\infty} \frac{z e^{-t\sqrt{z^2 + \alpha^2}}}{\sqrt{z^2 + \alpha^2}} \sin(\delta z) \mathop{dz},
\end{equation*}
valid for $t \ge 0$. Replacing in \eqref{integral_k1} and interchanging the order of integration, we have
\begin{align*}
F(x; \alpha, \beta, \mu, \delta) &= 1 - \frac{\alpha \delta e^{\delta \gamma}}{\pi}\frac{1}{\alpha\delta}\int_0^{\infty} \frac{z \sin(\delta z)}{\sqrt{z^2 + \alpha^2}}  \int_{x-\mu}^{\infty} e^{-t\sqrt{z^2 + \alpha^2} + \beta t} \mathop{dt} \mathop{dz}\\
&= 1 - \frac{e^{\delta \gamma}}{\pi}\int_0^{\infty} \frac{z \sin(\delta z)}{\sqrt{z^2 + \alpha^2}}  \frac{e^{-(x-\mu)\left(\sqrt{z^2 + \alpha^2} - \beta\right)}}{\sqrt{z^2 + \alpha^2} - \beta} \mathop{dz}
\end{align*}
where the inner integral converges for $x-\mu > 0$.
\end{proof}

Analogously, we obtain an integral representation in terms of the exponential integral $E_1(x)$ using the cosine transform of the modified Bessel function.
\begin{proposition}
An integral representation of the cumulative distribution function in terms of the exponential integral $E_1(x)$ is given by
\begin{equation}
F(x; \alpha, \beta, \mu, \delta) = 1 - \frac{\delta e^{\delta \gamma}}{\pi}\int_0^{\infty} E_1\left((x-\mu)(\sqrt{t^2 + \alpha^2} - \beta)\right) \cos(\delta t) \mathop{dt}, \quad x-\mu > 0.
\end{equation}
\end{proposition}
\begin{proof}
This result can be derived from the sine transform in \eqref{integral_sine_transform} using integration by parts or using the cosine transform \cite[\S 1.4]{Bateman1954}
\begin{equation*}
\frac{K_1\left(\alpha\sqrt{\delta^2 + t^2}\right)}{\sqrt{\delta^2 + t^2}} = \frac{1}{\alpha t}\int_0^{\infty} e^{-t\sqrt{z^2 + \alpha^2}} \cos(\delta z) \mathop{dz}.
\end{equation*}
Thus,
\begin{equation*}
F(x; \alpha, \beta, \mu, \delta) = 1 - \frac{\alpha \delta e^{\delta \gamma}}{\pi}\int_{x-\mu}^{\infty} \frac{e^{\beta t}}{\alpha t}\int_0^{\infty} e^{-t\sqrt{z^2 + \alpha^2}} \cos(\delta z) \mathop{dz} \mathop{dt},
\end{equation*}
where the inner integral is expressible in closed form in terms of the exponential integral $E_1(x)$
\begin{equation*}
\int_{x-\mu}^{\infty}\frac{1}{t}e^{-(\sqrt{z^2 + \alpha^2} - \beta)t} \mathop{dt} = E_1\left((x-\mu)(\sqrt{z^2 + \alpha^2} - \beta)\right).
\end{equation*}
\end{proof}

Moreover, an integral representation whose integrand decays double-exponentially is obtained using the transformation described in \cite[\S 42.4]{Temme2015}. First, we write $\beta = \alpha \tanh(\theta)$, valid since $|\beta| < \alpha$. Substituting in \eqref{integral_k1} the term $x-\mu = \delta \sinh(\theta + u)$, we obtain
\begin{equation}
F(x; \alpha, \beta, \mu, \delta) = \frac{\alpha \delta e^{\delta \gamma}}{\pi} \int_{-\infty}^{\tau} K_1 \left(\alpha \delta \cosh(\theta + u)\right) e^{\beta \delta \sinh(\theta + u)} \mathop{du},
\end{equation}
where
\begin{equation}
\tau = \arcsinh\left(\frac{x - \mu}{\delta}\right) - \theta.
\end{equation}

To conclude, a discussion on the numerical integration methods for effective computation of some the previous integrals is presented in Section \ref{section_numerical_integration}.

\section{Methods of computation}\label{section_methods_of_computation}
In this Section, we describe the methods used for efficient computation of the cumulative distribution function for the general case and various special cases, $\beta = 0$ and $x=\mu$, respectively. In particular, we explore series expansions, asymptotic expansions, uniform asymptotic expansions and numerical integration. We emphasize that many of the techniques and numerical methods introduced in the sections devoted to the two special cases shall also be used to develop expansions of the general case in Section \ref{section_general_case}.

\subsection{Special case $\beta = 0$}\label{section_special_case_beta_0}
As mentioned earlier, the case $\beta = 0$ corresponds to the symmetric case. The symmetric NIG distribution, but also the cases where $\beta \to 0$, has been widely used in financial modelling for modelling macro-economical factors as a fat-tailed alternative to the normal distribution. In this Section, we develop convergent series expansions, uniform asymptotic expansion and asymptotic expansions for this special case.

\subsubsection{Expansions $|x-\mu| \to 0$}
For developing a series expansion for the case $|x-\mu| \to 0$, we start from the integral representation in \eqref{integral_phi}, after expanding the square
\begin{equation}\label{integral_phi2}
F(x; \alpha, 0, \mu, \delta) = \frac{\delta e^{\delta \alpha}}{\sqrt{2\pi}} \int_0^{\infty} \Phi\left(\frac{x - \mu}{\sqrt{t}}\right) t^{-3/2} e^{-\frac{\delta^2}{2t} - \frac{\alpha^2}{2}t} \mathop{dt}.
\end{equation}
We proceed expanding the term $\Phi\left(\frac{x-\mu}{\sqrt{t}}\right)$, by using the two well-known absolutely convergent series expansions of $\Phi(x)$ \cite[\S 2]{Lebedev1972}
\begin{equation}\label{phi_expansion_1}
\Phi(x) = \frac{1}{2} + \frac{1}{\sqrt{2\pi}}\sum_{k=0}^{\infty} \frac{(-1)^k x^{2k + 1}}{2^k k! (2k+1)},
\end{equation}
and
\begin{equation}\label{phi_expansion_2}
\Phi(x) = \frac{1}{2} + \frac{e^{-x^2 / 2}}{\sqrt{2\pi}}\sum_{k=0}^{\infty} \frac{x^{2k+1}}{(2k + 1)!!}.
\end{equation}
If we choose the expansion \eqref{phi_expansion_1} and interchange the order of integration and summation, the resulting integral has the form
\begin{equation}\label{phi_expansion_integral_xmu_b0_alternating}
F(x; \alpha, 0, \mu, \delta) = \frac{\delta e^{\delta \alpha}}{\sqrt{2\pi}} \sum_{k=0}^{\infty} \frac{(-1)^k (x-\mu)^{2k + 1}}{2^k k! (2k + 1)}\int_0^{\infty} t^{-k-2} e^{-\frac{\delta^2}{2t} - \frac{\alpha^2}{2}t} \mathop{dt},
\end{equation}
where the integral has a closed-form in terms of the modified Bessel function. In general,
\begin{equation}\label{bessel_integral}
\int_0^{\infty} t^{\lambda - 1}e^{-a/t - zt} \mathop{dt} = 2\left(\frac{a}{z}\right)^{\lambda/2} K_{\lambda}(2\sqrt{a z}).
\end{equation}
Inserting \eqref{bessel_integral} in \eqref{phi_expansion_integral_xmu_b0_alternating} and rearranging terms, we obtain the the alternating series
\begin{equation}\label{expansion_xmu_b0_alternating}
F(x; \alpha, 0, \mu, \delta) = \frac{1}{2} + \frac{\delta e^{\delta \alpha}}{\pi}\sum_{k=0}^{\infty} \frac{(-1)^k (x-\mu)^{2k+1}}{2^k k! (2k + 1)} \left(\frac{\alpha}{\delta}\right)^{k+1}K_{k+1}(\alpha \delta).
\end{equation}
Moreover, to obtain a similar series with positive terms, we choose the expansion of $\Phi(x)$ in \eqref{phi_expansion_2}, yielding
\begin{equation}\label{expansion_xmu_b0_positive}
F(x; \alpha, 0, \mu, \delta) = \frac{1}{2} + \frac{\delta e^{\delta \alpha}} {\pi}\sum_{k=0}^{\infty} \frac{(x-\mu)^{2k+1}}{(2k +1)!!} \left(\frac{\alpha}{\omega}\right)^{k+1}K_{k+1}(\alpha \omega).
\end{equation}

Note that the series expansions \eqref{expansion_xmu_b0_alternating} and \eqref{expansion_xmu_b0_positive} can be written as a truncated series with the corresponding remainder term. For example, truncating the series \eqref{expansion_xmu_b0_positive} at $N$, we can write
\begin{equation}\label{expansion_xmu_b0_positive_error_term}
\sum_{k=0}^{\infty} T_k = \sum_{k=0}^{N-1} T_k + \sum_{k=N}^{\infty}T_k, \quad T_k = \left(\frac{(x-\mu)^2 \alpha}{\omega}\right)^k \frac{K_{k+1}(\alpha \omega)}{(2k +1)!!}.
\end{equation}
Thus, one has the series with remainder term $R_N = \sum_{k=N}^{\infty}T_k$
\begin{equation}\label{expansion_xmu_b0_positive_with_remainder}
F(x; \alpha, 0, \mu, \delta) = \frac{1}{2} + \frac{\delta e^{\delta \alpha}}{\pi}\frac{(x-\mu)\alpha}{\omega}\left(\sum_{k=0}^{N-1}T_k + R_N\right).
\end{equation}

For a rigorous evaluation of error bounds given a number of terms $N$, it is convenient to calculate an upper bound for $R_N$ in \eqref{expansion_xmu_b0_positive_error_term}. We refer to Remark \ref{remark_upper_bound_remainder} for a derivation involving an upper bound of the last omitted term $T_N$. The following lemma provides an upper bound for $K_{\nu+1}(x)$ required for a posterior derivation of an upper bound for the term $T_N$.
\begin{lemma}\label{lemma_1}
For $x \ge 0$ and $\nu \ge -\frac{1}{2}$ we have
\begin{equation}
K_{\nu + 1}(x) < \frac{\Gamma(\nu + 1)2^{\nu}}{x^{\nu + 1}}.
\end{equation}
\end{lemma}
\begin{proof}
The proof reduces to combining the uniform bound in \cite{Gaunt2016}
\begin{equation}
x K_{\nu + 1}(x) I_{\nu}(x) \le 1.
\end{equation}
with the lower bound \cite{Luke1972}
\begin{equation}
\left(\frac{x}{2}\right)^{\nu}\frac{1}{\Gamma(\nu + 1)} < I_{\nu}(x).
\end{equation}
Then, it follows that
\begin{equation}
K_{\nu + 1}(x) \le \frac{1}{x I_{\nu}(x)} < \frac{\Gamma(\nu + 1)2^{\nu}}{x^{\nu + 1}}.
\end{equation}
\end{proof}

\begin{theorem}\label{theorem_expansion_xmu_b0_postivie_remainder}
Given $\alpha > 0$, $\omega > 0$, $x-\mu \in \mathbb{R}$ and $N \in \mathbb{N}$, the first omitted term $T_N$ in \eqref{expansion_xmu_b0_positive_with_remainder} satisfies
\begin{equation}\label{bound_TN_xmu_b0_positive}
T_N < \frac{1}{2\alpha\omega}\left(\frac{x-\mu}{\omega}\right)^{2N} \sqrt{\frac{\pi}{N + 1/2}}.
\end{equation}
\end{theorem}

\begin{proof}
The use of Lemma \ref{lemma_1} gives
\begin{align*}
T_N &= \left(\frac{(x-\mu)^2 \alpha}{\omega}\right)^N \frac{K_{N+1}(\alpha \omega)}{(2N +1)!!}\\
&<  \left(\frac{(x-\mu)^2 \alpha}{\omega}\right)^N \frac{\Gamma(N + 1)2^N}{(\alpha\omega)^{N + 1}(2N +1)!!}\\
&= \frac{1}{\alpha\omega}\left(\frac{x-\mu}{\omega}\right)^{2N} \frac{(N! 2^N)^2}{(2N + 1)!}.
\end{align*}
An upper bound for the ratio of factorials in the previous inequality is given by
\begin{equation*}
\frac{(N! 2^N)^2}{(2N + 1)!} = \frac{\sqrt{\pi}}{2}\frac{\Gamma(N+1)}{\Gamma(N + 3/2)} \le \frac{1}{2}\sqrt{\frac{\pi}{N + 1/2}},
\end{equation*}
where we use the fact that $\Gamma(N + 3/2) = (N + 1/2) \Gamma(N + 1/2)$ and the upper bound of the ratio of gamma functions \cite{Wendel1948}
\begin{equation}
\frac{\Gamma(x + 1)}{\Gamma(x+s)} \le (x + s)^{1-s}, \quad s \in (0, 1).
\end{equation}
Thus, the following bound for $T_N$ holds
\begin{equation}
T_N < \frac{1}{2\alpha\omega}\left(\frac{x-\mu}{\omega}\right)^{2N} \sqrt{\frac{\pi}{N + 1/2}}.
\end{equation}
\end{proof}

To study the regime of applicability for the expansion, we can estimate the required number of terms $N$ equating the bound of $T_N$ in \eqref{bound_TN_xmu_b0_positive} times the normalizing factor, $\frac{\delta e^{\delta \alpha} (x-\mu)\alpha}{\pi \omega}$, with the requested absolute error $\epsilon$ and solving for $N$, which gives
\begin{equation}\label{N_expansion_xmu_b0_positive}
N \approx -\frac{\Re(W_{-1}(D))}{\log(A^4)} - \frac{1}{2}, \quad A = \frac{x-\mu}{\omega}, \quad B = \frac{A\delta e^{\delta \alpha}}{2\omega\pi}, \quad C = \frac{\epsilon^2}{B^2 \pi}, \quad D = -\frac{\log(A^4)}{C A^2}
\end{equation}
where $W_k(x)$ denotes the Lambert $W$ function \cite[\S 4.13]{NIST:DLMF}. The branch $k=-1$ is used to obtain the maximum real $N$. Note that since $A^2 < 1$ by definition, $D > 0$. When $C A^2$ is tiny, $W_{-1}(D)$ can be approximated as $\Re(W_{-1}(D)) \sim \log(D) - \log(\log(D))$ using the first two terms of the asymptotic expansion in \cite[\S 4.13.10]{NIST:DLMF}.

The previous analysis performed on the expansion \eqref{expansion_xmu_b0_positive} is repeated for the alternating expansion \eqref{expansion_xmu_b0_alternating}. The main results are summarized below for the purpose of brevity. The expansion \eqref{expansion_xmu_b0_alternating} rewritten including the remainder term follows
\begin{equation}
F(x; \alpha, 0, \mu, \delta) = \frac{1}{2} + \frac{(x-\mu)\alpha e^{\delta \alpha}}{\pi} \left( \sum_{k=0}^{N-1} T_k + R_N\right),  \quad T_k = \left(-\frac{(x-\mu)^2 \alpha}{\delta}\right)^k \frac{K_{k+1}(\alpha \delta)}{2^k k! (2k + 1)}.
\end{equation}
The last omitted term $T_N$ satisfies
\begin{equation}
|T_N| < \frac{1}{\alpha \delta} \left(\frac{x-\mu}{\delta}\right)^{2N} \frac{1}{2N + 1},
\end{equation} 
and the number of terms $N$ can be determined employing the Lambert $W$ function for a given error $\epsilon$
\begin{equation}
N \approx -\frac{\Re(W_{-1}(D))}{\log(A^2)} - \frac{1}{2}, \quad A = \frac{x-\mu}{\delta}, \quad B = \frac{A e^{\delta \alpha}}{\pi}, \quad C = \frac{\epsilon}{B}, \quad D = -\frac{\log(A)}{CA}.
\end{equation}

Now we are in position to compare both series and their respective domains of applicability. A first important observation is that the alternating series \eqref{expansion_xmu_b0_alternating} does not converge when $|x - \mu| > \delta$, and the number of terms $N$ increases rapidly when $|x - \mu| \sim \delta$. In contrast, the series \eqref{expansion_xmu_b0_positive} is absolutely convergent. The convergence of the latter can be reliably assessed applying to $T_k$ in \eqref{expansion_xmu_b0_positive_error_term} the asymptotic estimates of $K_{k+1}(\alpha\omega)$ for $\alpha \omega \to 0$ and $\alpha \omega \to \infty$, \eqref{besselk_x_to_0} and \eqref{besselk_x_to_inf}, respectively:
\begin{equation}\label{besselk_x_to_0}
K_{\nu}(x) \sim \frac{2^{\nu - 1} \Gamma(\nu)}{x^{\nu}}, \quad x \to 0, \quad \nu > 0,
\end{equation}
\begin{equation}\label{besselk_x_to_inf}
K_{\nu}(x) \sim \sqrt{\frac{\pi}{2x}} e^{-x}, \quad x \to \infty, \quad \nu \in \mathbb{R},
\end{equation}
obtaining
\begin{align*}
T_k &\sim \frac{\sqrt{\pi}}{2\alpha\omega} \left(\frac{x-\mu}{\omega}\right)^{2k}\frac{\Gamma(k + 1)}{\Gamma(k + 3/2)}, & \alpha\omega \to 0,\\
T_k &\sim \sqrt{\frac{\pi}{2\alpha \omega}} \left(\frac{(x-\mu)^2\alpha}{\omega}\right)^k \frac{e^{-\alpha \omega}}{(2k + 1)!!}, &\alpha\omega \to \infty.
\end{align*}
The asymptotic estimates of $T_k$ show that the series expansion \eqref{expansion_xmu_b0_positive} is slowly convergent when
\begin{equation*}
\left(\frac{x-\mu}{\omega}\right)^2 \to 1 \Longleftrightarrow \delta \to 0,
\end{equation*}
and the ratio of convergence improves when $\delta \to \infty$, then it can be viewed as an asymptotic expansion for large $\delta$. For $k \to \infty$, $T_k$ follows the asymptotic behaviour
\begin{equation*}
T_k \sim \frac{1}{\alpha\omega}\left(\frac{x-\mu}{\omega}\right)^{2k}\sqrt{\frac{\pi}{e}}\frac{1}{\sqrt{4k + 2}}\left(\frac{2k + 2}{2k + 1}\right)^{k + \frac{1}{2}}, \quad k \to \infty,
\end{equation*}
where we use the asymptotic estimate of modified Bessel function for large order
\begin{equation}\label{besselk_order_to_inf}
K_{\nu}(x) \sim \sqrt{\frac{\pi}{2\nu}}\left(\frac{ex}{2\nu}\right)^{-\nu}, \quad \nu \to \infty, \quad x \neq 0,
\end{equation}
and apply Stirling's approximation for the double factorial. It remains to analyze the required number of terms for different parameters. Table \ref{table_bound_remainder_xmu_b0_positive} shows the estimated $N$ to achieve machine-precision using \eqref{N_expansion_xmu_b0_positive}. As a remark, the estimation of $N$ for large $\alpha\omega$ can be enhanced by selecting $N$ using the asymptotic estimate for $\alpha\omega \to \infty$ via binary search. Table \ref{table_bound_remainder_xmu_b0_positive} also shows that for small $\delta$ and $\alpha\omega$, the required number of terms makes the series expansion impractical. In the next section, we present a convergence acceleration method to obtain a rapidly convergent series for these cases.

\begin{table}[htp]
	\centering
	\scalebox{0.9}{
\begin{tabular}{ccccc|cc}
\hline
$x$ & $\alpha$ & $\mu$ & $\delta$ & $\alpha\omega$ & $N$ \eqref{N_expansion_xmu_b0_positive} &  $R_N$\\
\hline
	1 & 5 & 1/4 & 1 & 6.25 & 42 & $8.8\cdot 10^{-19}$\\
	1/2 & 1/3 & 1/4 & 1/10 & 0.09 & 236 & $2.9\cdot 10^{-17}$\\
    1/3 & 10 & 1/5 & 1/50 & 1.35 & 1,494 & $2.1\cdot 10^{-16}$\\
    1 & 10 & 1/5 & 5 & 50.64 &  25 & $1.9\cdot 10^{-29}$\\
    3 & 10 & 1/5 & 10 & 103.85 & 53 & $2.4\cdot 10^{-36}$\\
    10 & 1/10 & 1/5 & 10 & 1 & 53 & $3.8\cdot 10^{-18}$\\
\hline
\end{tabular}}
	\caption{The remainder of the series expansion  \eqref{expansion_xmu_b0_positive} and estimated $N$ using \eqref{N_expansion_xmu_b0_positive} to achieve machine precision.}
	\label{table_bound_remainder_xmu_b0_positive}
\end{table}

\subsubsection{Convergence acceleration of the expansion $|x-\mu| \to 0$}
Table \ref{table_bound_remainder_xmu_b0_positive} showed that for small values of $\delta$ and $\alpha \omega$ the required number of terms grows considerably, thereby, resorting to numerical integration (see Section \ref{section_numerical_integration}) might be a more efficient approach. Alternatively, a common technique to reduce the number of terms of slowly convergent series is to use series acceleration methods (Shank's transformation and Levin-type transformations among others \cite[\S 9]{Gil2007}). 
Attempts to use Shank's transformation were unsuccessful; we did not observe a significant reduction in the number of terms $N$ while achieving only around ten correct digits systematically for all cases. In addition, a major drawback is that Shank's acceleration requires higher-precision arithmetic to compensate for cancellation effects, which makes us discard this option for an implementation in double-precision arithmetic. 

In the following, we investigate the use of exponentially improved asymptotic expansions. When the information about the remainder is available, this technique consists of re-expanding the remainder, obtaining an expansion exhibiting faster convergence. For further details, we refer to \cite[\S 14]{Olver1997}. Consider the convergent series expansion in \eqref{expansion_xmu_b0_positive_error_term}
\begin{equation}
F(x;\alpha, 0, \mu, \delta) = \frac{1}{2} + \frac{\delta e^{\delta \alpha}}{\pi} \frac{(x-\mu)\alpha}{\omega} \left(\sum_{k=0}^{N-1} T_k + \sum_{k=N}^{\infty} T_k\right),
\end{equation}
where
\begin{equation}
T_k = \frac{K_{k+1}(\alpha\omega)}{(2k + 1)!!} z^k, \quad z = \frac{(x-\mu)^2 \alpha}{\omega}.
\end{equation}
and remainder $R_N = \sum_{k=N}^{\infty} T_k$. An integral representation of $R_N$ can be obtained using Basset's integral \cite[\S 10.32.11]{NIST:DLMF} representation of the modified Bessel function, defined as
\begin{equation}\label{basset_integral}
K_{k+1}(\alpha\omega) = \frac{\Gamma(k + 3/2)}{\sqrt{\pi}}\left(\frac{2\alpha}{\omega}\right)^{k+1} \int_0^{\infty} \frac{\cos(\omega t)}{(t^2 + \alpha^2)^{k + 3/2}} \mathop{dt}.
\end{equation}
Basset's integral is chosen to transform the term $\alpha\omega$ into $\alpha/\omega$. Thus, inserting \eqref{basset_integral} into the remainder in \eqref{expansion_xmu_b0_positive_error_term}, we have an integral representation of $R_N$
\begin{align}\label{convergent_remainder_integral}
R_N &= \frac{1}{\sqrt{\pi}} \left(\frac{2\alpha}{\omega}\right) \int_0^{\infty} \frac{\cos(\omega t)}{(t^2 + \alpha^2)^{3/2}} \left[\sum_{k=N}^{\infty} \left(\frac{2z\alpha}{\omega (t^2 + \alpha^2)}\right)^k \frac{\Gamma(k + 3/2)}{(2k + 1)!!}\right] \mathop{dt}\nonumber\\
&= C \int_0^{\infty} \frac{\cos(\omega t)}{(t^2 + \alpha^2)^{N + 3/2}} \frac{1}{\left(2 - \frac{m}{t^2+ \alpha^2}\right)} \mathop{dt}.
\end{align}
where, for the purpose of brevity, we use
\begin{equation}
m = \frac{2z\alpha}{\omega} = 2\left(\frac{(x-\mu)\alpha}{\omega}\right)^2,  \quad C = \frac{2}{\sqrt{\pi}} \left(\frac{2\alpha}{\omega}\right)\frac{\Gamma(N + 3/2)}{(2N + 1)!!} m^N.
\end{equation}

Now, we expand the term $\cos(\omega t)$, recall that $\omega = \sqrt{\delta^2 + (x-\mu)^2}$, yielding
\begin{equation}\label{convergent_remainder_series_integral}
R_N = C \sum_{k=0}^{\infty} \frac{(-1)^k \omega^{2k}}{(2k)!} \int_0^{\infty} \frac{t^{2k}}{(t^2 + \alpha^2)^{N + 3/2}} \frac{1}{\left(2 - \frac{m}{t^2+ \alpha^2}\right)} \mathop{dt},
\end{equation}
and note that the integrals above converge for $k \le N$. Next, after various iterations with the assistance of Mathematica \cite{Mathematica}, we arrive to a closed-form expression for the previous integral in terms of the Gauss hypergeometric function $_2F_1(a, b; c; z)$ \cite[\S 15]{NIST:DLMF}
\begin{equation}\label{convergent_remainder_inner_integral}
\int_0^{\infty} \frac{t^{2k}}{(t^2 + \alpha^2)^{N + 3/2}} \frac{1}{\left(2 - \frac{m}{t^2+ \alpha^2}\right)} \mathop{dt} = P_k + Q_k,
\end{equation}
with
\begin{align}
P_k &= \frac{2^{N-2}\alpha^{2(k - N - 1)} \Gamma(N + 1 - k) \Gamma(k - 1/2)}{\sqrt{\pi} (2N + 1)!!} \, _2F_1\left(1, N + 1 - k; \frac{3}{2} -k, 1 - \frac{m}{2\alpha^2}\right),\\
Q_k &= \frac{2^{N-1-k}(2\alpha^2 - m)^{k - 1/2} \pi \sec(k\pi)}{m^{N + 1/2}}.
\end{align}
Subsequently, substituting \eqref{convergent_remainder_inner_integral} in \eqref{convergent_remainder_series_integral}, we write $R_N$ as the sum of two series $R_N  = S_P + S_Q$ given by
\begin{align}\label{convergent_remainder_series_P}
S_P &= \frac{(2m)^N }{\alpha^{2N + 1} \omega\pi} \frac{\Gamma(N + 3/2)}{(2N + 1)!! (2N - 1)!!} \nonumber\\
&\times \sum_{k=0}^{N} \frac{(-1)^k (\alpha\omega)^{2k}}{(2k)!} \Gamma(N + 1 - k) \Gamma(k - 1/2) \, _2F_1\left(1, N + 1 - k; \frac{3}{2} -k, 1 - \frac{m}{2\alpha^2}\right)
\end{align}
and 
\begin{align}
S_Q &= C \frac{2^{N-1}}{m^{N + 1/2}\sqrt{2\alpha^2 - m}} \pi \sum_{k=0}^{\infty} \frac{(-1)^k}{(2k)!} \left(\frac{\omega^2(2\alpha^2 - m)}{2}\right)^k \sec(k\pi)\nonumber\\
&= \frac{\Gamma(N + 3/2) 2^{N+1}}{(2N + 1)!!} \frac{\alpha}{\omega}\sqrt{\frac{\pi}{m(2\alpha^2 - m)}}\cosh\left(\sqrt{\frac{\omega^2 (2\alpha^2 - m)}{2}}\right).
\end{align}
The sum $S_P$ is terminating at $k=N$ due to the term $\Gamma(N + 1 - k)$ in the series. Thus, the resulting series to compute $F(x; \alpha, 0, \mu, \delta)$ is
\begin{equation}\label{convergent_accelerated_series_xmu_b0_positive}
F(x;\alpha, 0, \mu, \delta) = \frac{1}{2} + \frac{\delta e^{\delta \alpha}}{\pi} \frac{(x-\mu)\alpha}{\omega} \left(\sum_{k=0}^{N-1} T_k + S_P + S_Q\right).
\end{equation}
Next, we focus on determining the optimal $N$ for a desired precision $\epsilon$. The smallest term in $S_P$ occurs when $k=N$, corresponding to
\begin{equation}
S_P = \frac{(-1)^N (\alpha\omega)^{2N}}{(2N)!} \Gamma(N - 1/2) \, _2F_1\left(1, 1; \frac{3}{2} -N, 1 - \frac{m}{2\alpha^2}\right).
\end{equation}
Moreover, inspecting the argument of $_2F_1$ in \eqref{convergent_remainder_series_P},
\begin{equation}
1 - \frac{m}{2\alpha^2} = 1 - \left(\frac{x-\mu}{\omega}\right)^2 < 1,
\end{equation}
we see that for $\alpha \omega \to 0$, $\delta \to 0$, the argument is close to 1. Therefore, for the parameter regime of interest, the last term can be effectively approximated taking the limit $\lim_{x \to 0}\, _2F_1\left(1, 1, \frac{3}{2}-k, x\right) = 1$. Then, it remains to solve the following equation for $N$ 
\begin{equation}
\frac{(\sqrt{2m} \omega)^{2N}}{\alpha\omega \pi} \frac{\Gamma(N - 1/2)\Gamma(N + 3/2)}{(2N)!(2N + 1)!! (2N - 1)!!} = \epsilon.
\end{equation}
An asymptotic approximation follows by using Stirling approximation of the ratio of gamma functions and double factorials
\begin{equation}
\frac{\Gamma(N - 1/2)\Gamma(N + 3/2)}{(2N)!(2N + 1)!! (2N - 1)!!} \sim \frac{\sqrt{\pi}}{4} \left(\frac{e}{N}\right)^{2N} 2^{-4N}, \quad N \to \infty.
\end{equation}
obtaining the simplified equation
\begin{equation}
(\sqrt{2m} \omega)^{2N} \left(\frac{e}{N}\right)^{2N} 2^{-4N} = 4\sqrt{\pi}\alpha\omega\epsilon.
\end{equation}
The solution of the latter permits a closed-form in terms of principal branch of the Lambert $W$ function
\begin{equation}\label{N_expansion_xmu_acc}
N \approx - \frac{\log(A)}{W_0\left(-\frac{\log(A)}{2 e \sqrt{2m} \omega} \right)}, \quad A = 4\sqrt{\pi}\alpha\omega\epsilon.
\end{equation}

Table \ref{table_bound_remainder_xmu_b0_positive_large_N} shows the estimated number of terms of the accelerated convergent series \eqref{convergent_accelerated_series_xmu_b0_positive}, $N_{acc}$, and the actual number of terms $N^*$ to achieve machine-precision for small values of $\alpha\omega$ and $\delta$. The first and most notable observation is the reduction in the number of terms compared with the convergent series \eqref{N_expansion_xmu_b0_positive}, especially for $\alpha \le 1$. The second observation is the accuracy of the estimate in $N_{acc}$ in \eqref{N_expansion_xmu_acc}, which only seems to underestimate $N$ for large $\alpha\omega$ slightly.

\begin{table}[htp]
	\centering
	\scalebox{0.9}{
\begin{tabular}{ccccc|ccccc}
\hline
$x$ & $\alpha$ & $\mu$ & $\delta$ & $\alpha\omega$ & $N$ \eqref{N_expansion_xmu_b0_positive} & $N_{acc}$ \eqref{N_expansion_xmu_acc} & error & $N^*$ & error\\
\hline
	1 & 50 & 1/5 & 1/3 & 43.33 & 324 & 69 & $2.9\cdot 10^{-13}$ & 73 & $1.1\cdot 10^{-17}$\\
    2 & 5 & 1/5 & 1/10 & 9.01 & 10,242 & 25 & $4.6\cdot 10^{-21}$ & 22 & $1.1 \cdot 10^{-16}$\\
	1 & 1/10 & 1/5 & 1/100 & 0.08 & 179,715 & 5 & $1.3 \cdot 10^{-21}$ & 4 & $2.4 \cdot 10^{-17}$\\
	5 & 1 & 1/5 & 1/100 & 4.8 & 5,645,686 & 18 & $1.1 \cdot 10^{-22}$ & 15 & $1.5 \cdot 10^{-17}$\\
	20 & 1/100 & 1/5 & 1/100 & 0.198 & 84,914,922 & 6 & $1.1 \cdot 10^{-22}$ & 5 & $4.5 \cdot 10^{-19}$\\
	\hline
\end{tabular}}
	\caption{Absolute error estimating $N$ using \eqref{N_expansion_xmu_b0_positive} for the series expansion  \eqref{expansion_xmu_b0_positive} with machine-precision absolute error.}
	\label{table_bound_remainder_xmu_b0_positive_large_N}
\end{table}
\begin{remark}\label{remark_upper_bound_remainder}
A simple bound for $R_N$ can be derived from the integral representation in \eqref{convergent_remainder_integral} as follows
\begin{align*}
R_N &= C \int_0^{\infty} \frac{\cos(\omega t)}{(t^2 + \alpha^2)^{N + 3/2}} \frac{1}{\left(2 - \frac{m}{t^2+ \alpha^2}\right)} \mathop{dt}\\
&\le \frac{C}{\left(2 - \frac{m}{\alpha^2}\right)} \int_0^{\infty} \frac{\cos(\omega t)}{(t^2 + \alpha^2)^{N + 3/2}} \mathop{dt}\\
&= \frac{1}{1 - \left(\frac{x-\mu}{\omega}\right)^2} T_N
\end{align*}

using the upper bound for $T_N$ in \eqref{bound_TN_xmu_b0_positive}, we obtain
\begin{equation}
R_N < \frac{1}{1 - \left(\frac{x-\mu}{\omega}\right)^2} \frac{1}{2\alpha\omega}\left(\frac{x-\mu}{\omega}\right)^{2N} \sqrt{\frac{\pi}{N + 1/2}}.
\end{equation}
\end{remark}

\subsubsection{Expansion $|x-\mu| \to \infty$}
To derive an asymptotic expansion for large $|x-\mu|$, we use the expansion derived in \eqref{phi_expansion_incgamma_t_small} as a starting point, taking $a = x-\mu$ and $b = 0$. Similar to the derivation of the expansion in \eqref{expansion_xmu_b0_alternating}, we interchange summation and integration, and the resulting integral is expressible in closed-form as a modified Bessel function, see Equation \eqref{bessel_integral}. Thus, for $x - \mu < 0$ we have the alternating asymptotic expansion
\begin{align}\label{asymptotic_expansion_xmu_b0}
F(x; \alpha, 0, \mu, \delta) &\sim \frac{\delta e^{\delta \alpha}}{2\pi} \sum_{k=0}^{\infty} \frac{(-1)^{k+1}}{k!} \frac{\Gamma(2k + 1)}{2^k (x-\mu)^{2k + 1}} \int_0^{\infty} t^{k-1} e^{-\frac{\omega^2}{2t} - \frac{\alpha^2}{2}t} \mathop{dt} \nonumber \\
&= -\frac{\delta e^{\delta \alpha}}{\pi (x-\mu)} \sum_{k=0}^{\infty}\frac{(-1)^k\Gamma(2k + 1)}{k!} \left(\frac{\omega}{2(x-\mu)^2\alpha}\right)^k K_k(\alpha \omega).
\end{align}
Note that the remainder is bounded in magnitude by the first neglected term. Moreover, the convergence of the expansion improves for large $\alpha$ and $\alpha / \omega > 1$. Finally, for the case $x - \mu > 0$ we use the reflection formula \eqref{cdf_mirror}.

\subsubsection{Uniform expansion $\alpha \to \infty$, $\alpha \sim \delta$ and $|x-\mu| \gg 0$}\label{uniform_expansion_alpha_large}
For large $\alpha$, we consider the uniform asymptotic expansion in terms of modified Bessel functions described in \cite{Temme1990c} and \cite[\S 27]{Temme2015}. We write the Laplace-type integral \eqref{integral_phi} in the standard form
\begin{equation*}
F_{\lambda}(z, r) = C\int_0^{\infty} t^{\lambda - 1} e^{-z\left(t + r^2/t\right)} f(t) \mathop{dt},
\end{equation*}
where $C$ is a normalizing constant, $\lambda=-1/2$, $z = \alpha^2/2$, $r=\delta/\alpha$ and $f(t) = \Phi((x-\mu)/\sqrt{t})$. The saddle point of $e^{-z\left(t + r^2/t\right)}$ occurs at $\pm r$, but only the positive saddle point $r$ lies inside the interval of integration. Thus, we expand $f(t)$ at the saddle point $r$
\begin{equation*}
f(t) = \sum_{k=0}^{\infty} c_k(r)(t-r)^k,
\end{equation*}
after interchanging the order of summation and integration, we obtain 
\begin{equation}
F_{\lambda}(z, r) \sim \frac{1}{z^{\lambda}} \sum_{k=0}^{\infty} \frac{c_k(r) Q_k(\zeta)}{z^k}, \quad z \to \infty,
\end{equation}
where
\begin{equation*}
Q_k(\zeta) = \zeta^{\lambda + k}\int_0^{\infty} t^{\lambda - 1}(t-1)^k e^{-\zeta(t + 1/t)} \mathop{dt}, \quad \zeta = rz.
\end{equation*}

For $f(t) = \Phi((x-\mu)/\sqrt{t})$ the coefficients at $t=r$ satisfy the recurrence in \eqref{phi_expansion_at_u} setting $a=x-\mu$ and $b=0$. In particular, the recurrence can be simplified as follows
\begin{equation}
c_0(r) = \Phi\left(\frac{x-\mu}{\sqrt{r}}\right), \quad c_1(r) = -\frac{(x-\mu)}{2 r^{3/2}} \phi\left(\frac{x-\mu}{\sqrt{r}}\right)
\end{equation}
and
\begin{equation}
c_k(r) = \frac{(k - 1) ((x-\mu)^2 - 4 r(k-2) - 3r) c_{k-1}(r) - (2(k-2)^2 + k-2) c_{k-2}(r)}{2r^2 (k-1) k}, \quad k\ge 2.
\end{equation}
The functions $Q_k(\zeta)$ can be expressed as a binomial sum of modified Bessel functions, and satisfy the recurrence relation \cite[\S 27.3.28]{Temme2015}
\begin{equation}
Q_{k+2}(\zeta) = \left(k + \frac{1}{2} -2\zeta\right) Q_{k+1}(\zeta) + \zeta\left(2k + \frac{1}{2}\right)Q_k(\zeta) + k\zeta^2 Q_{k-1}(\zeta), \quad k\ge 1,
\end{equation}
with initial values
\begin{equation}
Q_0(\zeta) = \frac{2}{\sqrt{\zeta}} K_{\frac{1}{2}}(2 \zeta), \quad Q_1(\zeta) = 0, \quad Q_2(\zeta) = 2 \zeta^{3/2} \left(K_{\frac{3}{2}}(2\zeta) - K_{\frac{1}{2}}(2 \zeta)\right),
\end{equation}
where the special case $K_{n+1/2}(z)$, $n \in \mathbb{N}$, is a terminating sum of elementary functions requiring $n$ terms
\begin{equation}\label{besselk_half}
K_{n + 1/2}(z) = \sqrt{\frac{\pi}{2z}} \sum_{j=0}^n \frac{(n + j)!}{j! (n-j)!} (2z)^{-j}e^{-z}, \quad z\in \mathbb{C}.
\end{equation}
In particular the cases $n=0$ and $n=1$ are
\begin{equation*}
K_{\frac{1}{2}}(z) = \sqrt{\frac{\pi}{2z}}e^{-z}, \quad K_{\frac{3}{2}}(z) = \sqrt{\frac{\pi}{2 z}}\frac{e^{-z} (z+1)}{z},
\end{equation*}
and subsequent terms can be computed via recursion. Note, however, that the recurrence is numerically unstable, although the loss of accuracy due to instability is acceptable because only a limited number of terms is required when the expansion is applicable.
Thus, rearranging terms we have
\begin{equation}\label{uniform_expansion_a_large}
F(x; \alpha, 0, \mu, \delta) \sim \frac{\alpha \delta e^{\delta \alpha}}{2 \sqrt{\pi}} \sum_{k=0}^{\infty} \frac{2^k c_k\left(\frac{\delta}{\alpha}\right) Q_k\left(\frac{\alpha \delta}{2}\right)}{\alpha^{2k}}, \quad \alpha \to \infty.
\end{equation}
The expansion \eqref{uniform_expansion_a_large} is a uniform expansion as $\alpha \to \infty$, uniformly with respect to $\delta / \alpha > 0$. Note that large values of $\delta$ improves the rate of convergence of the expansion, as observed taking well-know asymptotic estimates for large argument of the modified Bessel function. We remark that expansions \eqref{expansion_xmu_b0_alternating} and \eqref{expansion_xmu_b0_positive} are also adequate for large values of $\alpha$ and $\delta$, but unlike the present expansion, the number of terms increases significantly when $|x-\mu| \gg 0$. Finally, we refer to \cite[\S 3.2]{Temme1990c} for details on error bounds and interpretation of the expansion.

\subsection{Special case $x = \mu$}
The second special case of interest, $x = \mu$, is arguably less common and has fewer applications, but it is nonetheless required to obtain a robust implementation. In particular, the attentive reader might have noticed that none of the series expansions previously introduced are suitable for this specific case.

First, consider the case $x = \mu$ and $\beta = 0$, corresponding to the case where the distribution is symmetric and centred at $x = \mu$, consequently it follows that
\begin{equation}
F(\mu; \alpha, 0, \mu, \delta) = \frac{1}{2}.
\end{equation}

Secondly, for the case $\beta \neq 0$, the integral representation of this special case is given by simple substitution in \eqref{integral_k1}
\begin{equation}\label{integral_k1_x=mu}
F(\mu; \alpha, \beta, \mu, \delta) = \frac{\alpha \delta e^{\delta \gamma}}{\pi} \int_{-\infty}^{0} \frac{K_1\left(\alpha\sqrt{\delta^2 + t^2}\right)}{\sqrt{\delta^2 + t^2}} e^{\beta t} \mathop{dt}.
\end{equation}

Using the series expansion of the exponential function and interchanging the order of integration and summation in \eqref{integral_k1_x=mu} gives that
\begin{equation*}
F(\mu; \alpha, \beta, \mu, \delta) = 1 - \frac{\alpha \delta e^{\delta \gamma}}{\pi} \sum_{k=0}^{\infty}\frac{\beta^k}{k!}\int_{0}^{\infty} t^k \frac{K_1\left(\alpha\sqrt{\delta^2 + t^2}\right)}{\sqrt{\delta^2 + t^2}} \mathop{dt}.
\end{equation*}
The integral is expressible in closed form using \cite[\S 6.596]{gradshteyn2007}
\begin{equation*}
\int_0^{\infty} t^k \frac{K_1\left(\alpha\sqrt{\delta^2 + t^2}\right)}{\sqrt{\delta^2 + t^2}} \mathop{dt} = \frac{2^{\frac{k-1}{2}} \Gamma\left(\frac{k+1}{2}\right)}{\alpha^{\frac{k+1}{2}} \delta^{\frac{-k+1}{2}}}K_{\frac{k-1}{2}}(\alpha\delta),
\end{equation*}
and rearranging terms and using the connection formula $2^{k/2}  \Gamma\left(\frac{k+1}{2}\right)/ k! = \sqrt{\pi} 2^{-k/2} / \Gamma\left(\frac{k}{2} + 1\right)$ yields
\begin{align}
F(\mu; \alpha, \beta, \mu, \delta) &= 1 - \sqrt{\frac{\alpha \delta}{2\pi}}e^{\delta \gamma} \sum_{k=0}^{\infty}\frac{\beta^k}{2^{\frac{k}{2}}\Gamma\left(\frac{k}{2} + 1\right)} \left(\frac{\delta}{\alpha}\right)^{\frac{k}{2}} K_{\frac{k-1}{2}}(\alpha \delta)\\
&=\sqrt{\frac{\alpha \delta}{2\pi}}e^{\delta \gamma} \sum_{k=0}^{\infty}\frac{(-\beta)^k}{2^{\frac{k}{2}}\Gamma\left(\frac{k}{2} + 1\right)} \left(\frac{\delta}{\alpha}\right)^{\frac{k}{2}} K_{\frac{k-1}{2}}(\alpha \delta).
\end{align}
The resulting expansions are convergent and the latter series expansion is preferred for $\beta < 0$ to avoid cancellation errors. A more rapidly convergent expansion for $\delta < \gamma$ or large values of $\delta$ and $\gamma$ can be obtained using the integral \eqref{integral_phi} and expanding the term $\Phi(-\beta\sqrt{t})$. Replacing $\Phi(-\beta\sqrt{t})$ in the integral \eqref{integral_phi} with the expansion \eqref{phi_expansion_1} and interchanging the order of integration and summation, we obtain
\begin{equation}\label{expansion_x=mu_pre}
F(\mu; \alpha, \beta, \mu, \delta) = \frac{1}{2} + \frac{\delta e^{\delta \gamma}}{2\pi}\sum_{k=0}^{\infty}\frac{(-1)^k (-\beta)^{2k+1}}{2^k k! (2k+1)} \int_0^{\infty} t^{k - 1} e^{-\frac{\delta^2}{2t} - \frac{\gamma^2}{2}t} \mathop{dt},
\end{equation}
where we can express the integral in terms of the modified Bessel function \eqref{bessel_integral}. Plugging in \eqref{expansion_x=mu_pre}, now yields the alternating series
\begin{equation}\label{series_x=mu_1}
F(\mu; \alpha, \beta, \mu, \delta) = \frac{1}{2} + \frac{\delta e^{\delta \gamma}}{\pi} \sum_{k=0}^{\infty} \frac{(-1)^k (-\beta)^{2k+1}}{2^k k! (2k + 1)} \left(\frac{\delta}{\gamma}\right)^k K_k(\gamma \delta).
\end{equation}
Similarly, using \eqref{phi_expansion_2} yields
\begin{equation}\label{series_x=mu_2}
F(\mu; \alpha, \beta, \mu, \delta) = \frac{1}{2} + \frac{\delta e^{\delta \gamma}}{\pi} \sum_{k=0}^{\infty} \frac{(-\beta)^{2k+1}}{(2k + 1)!!} \left(\frac{\delta}{\alpha}\right)^k K_k(\alpha \delta).
\end{equation}
The latter expansion being more convenient since $\alpha \ge \gamma$. Finally, to complement the series expansion, a special case of the asymptotic expansion defined in Equation \eqref{expansion_x_eq_mu_large_delta}, detailed in the next section, can be employed for large values of $\delta$ and large $\alpha$, and its convergence improves as $\beta \sim \alpha$.

\subsection{General case}\label{section_general_case}
In the general case, we continue using similar techniques to devise a set of series expansions and asymptotic expansions. Two of the series expansions involving Hermite polynomials make use of the two special cases studied in detail in the preceding sections.

\subsubsection{Expansion $|x-\mu| \to 0$ in terms of modified Bessel functions}
We start deriving a series expansion for small $|x-\mu|$ following similar steps to those proposed for the special case $\beta = 0$. In the same manner, we use the expansion of $\Phi(x)$ in  \eqref{phi_expansion_2} to obtain
\begin{equation}\label{general_series_xmu_zero_integral}
F(x;\alpha, \beta, \mu, \delta) = \frac{1}{2} + \frac{\delta e^{\delta \gamma}}{2\pi} \sum_{k=0}^{\infty} \frac{1}{(2k + 1)!!}I_k,
\end{equation}
\begin{equation}\label{general_xmu_zero_integral}
I_k = \int_0^{\infty} \left(\frac{x - (\mu +\beta t)}{\sqrt{t}}\right)^{2k+1} t^{-3/2} e^{-\frac{\delta^2}{2t} - \frac{\gamma^2}{2}t} e^{- \frac{(x-(\mu + \beta t)^2}{2t}} \mathop{dt}.
\end{equation}
The series expansion written in a more standard form is given by
\begin{equation}\label{general_series_xmu_zero_integral_2}
F(x;\alpha, \beta, \mu, \delta) = \frac{1}{2} + \frac{\delta e^{\delta \gamma + (x -\mu)\beta}}{2\pi}\sum_{k=0}^{\infty} \frac{(-\beta)^{2k+1}}{(2k + 1)!!}\int_0^{\infty} \left(t - \frac{x-\mu}{\beta}\right)^{2k+1} t^{-k-2} e^{-\frac{\omega^2}{2t} - \frac{\alpha^2}{2}t} \mathop{dt},
\end{equation}
where we denote $J_k$ the integral just introduced
\begin{equation}
J_k = \int_0^{\infty} \left(t - \frac{x-\mu}{\beta}\right)^{2k+1} t^{-k-2} e^{-\frac{\omega^2}{2t} - \frac{\alpha^2}{2}t} \mathop{dt}.
\end{equation}
The integral $J_k$ is expressible in terms of a finite series of modified Bessel functions after applying the binomial expansion
\begin{equation}
J_k = 2 \left(\frac{-(x-\mu)}{\beta}\right)^{2k+1} \left(\frac{\alpha}{\omega}\right)^{k+1} \sum_{j=0}^{2k+1} (-1)^j \binom{2k+1}{j} \left(\frac{\omega \beta}{\alpha (x-\mu)}\right)^j K_{k + 1 - j}(\alpha \omega).
\end{equation}
Thus, we have the series expansion
\begin{equation}\label{general_expansion_xmu_small_bessel}
F(x;\alpha, \beta, \mu, \delta) = \frac{1}{2} + \frac{(x-\mu)\alpha\delta e^{\delta \gamma + (x -\mu)\beta}}{\pi \omega}\sum_{k=0}^{\infty} \frac{A_k}{(2k+1)!!}\left(\frac{(x-\mu)^2\alpha}{\omega}\right)^k
\end{equation}
\begin{equation}
A_k = \sum_{j=0}^{2k+1} (-1)^j \binom{2k+1}{j} \left(\frac{\omega \beta}{\alpha (x-\mu)}\right)^j K_{k + 1 - j}(\alpha \omega)
\end{equation}
Observe that the convergence of the series improves for large values of $\alpha$ and $\delta$, hence it can be seen as a uniform asymptotic expansion with respect to these parameters. The first three coefficients $A_k$ are
\begin{align*}
A_0 &= -u K_0(z) + K_1(z)\\
A_1 &= 3u^2 K_0(z) - u(3 + u^2) K_1(z) + K_2(z)\\
A_2 &= -10u^3 K_0(z) + 5u^2(2 + u^2) K_1(z) - u(5 + u^4) K_2(z) + K_3(z)
\end{align*}
where, to compact notation, we use $u$ and $z$ defined as
\begin{equation}
u = \frac{\omega \beta}{\alpha (x-\mu)}, \quad z = \alpha \omega.
\end{equation}

Furthermore, the series expansion can be written as a truncated series with remainder. Using the representation in \eqref{general_series_xmu_zero_integral_2} we have a series in the form
\begin{equation}
\sum_{k=0}^{\infty} T_k = \sum_{k=0}^{N-1}T_k + \sum_{k=N}^{\infty} T_k, \quad T_k = \frac{(-\beta)^{2k+1}}{(2k+1)!!} I_k.
\end{equation}
The remainder, $R_N = \sum_{k=N}^{\infty} T_k$, integral representation is obtained as
\begin{align}\label{general_xmu_remainder_integral}
R_N &= \int_0^{\infty} e^{-\frac{\omega^2}{2t} - \frac{\alpha^2}{2}t} \left[ \sum_{k=N}^{\infty}\frac{(-\beta)^{2k+1}}{(2k+1)!!} \left(t - \frac{x-\mu}{\beta}\right)^{2k+1} t^{-k-2}\right] \mathop{dt}\\
&= \frac{(-1)^{2N + 1} 2^{N - 1/2} (2N + 1)}{(2N + 1)!!} e^{-\beta(x-\mu)} \int_0^{\infty} t^{-3/2} e^{-\frac{\delta^2}{2t} - \frac{\gamma^2}{2}t} \gamma\left(N + 1/2 , \frac{(\beta t - (x-\mu))^2}{2t}\right) \mathop{dt},
\end{align}
where $\gamma(a, z)$ is the lower incomplete gamma function. 

\subsubsection{Expansion $|x-\mu| \to 0$ in terms of Hermite polynomials}
Furthermore, we obtain an additional series expansion in terms of Hermite polynomials and the special case $x=\mu$, $F(\mu; \alpha, \beta, \mu, \delta)$. As in previous derivations, our starting point is the substitution of the term $\Phi\left(\frac{x - (\mu +\beta t)}{\sqrt{t}}\right)$ in \eqref{integral_phi}, this time,
by the Hermite-type expansion derived in \eqref{phi_expansion_a_small}. After inserting the expansion we have
\begin{equation}
F(x; \alpha, \beta, \mu, \delta) = F(\mu; \alpha, \beta, \mu, \delta) + \frac{\delta e^{\delta \gamma}}{2\pi} \sum_{k=0}^{\infty}\frac{(-1)^k (x-\mu)^{k+1}}{(k + 1)! 2^{k/2}}\int_0^{\infty} t^{-k/2 - 2} H_k\left(-\beta\sqrt{\frac{t}{2}}\right) e^{-\frac{\delta^2}{2t} - \frac{\alpha^2}{2}t} \mathop{dt}.
\end{equation}
Denote $I_k$ the integral in the series coefficients. The finite power series of $H_k(x)$ \cite[\S 18.5.13]{NIST:DLMF} allows the representation of $I_k$ as a finite series of modified Bessel functions. In this way, we obtain
\begin{align}
I_k  &= \int_0^{\infty} t^{-k/2 - 2} H_k\left(-\beta\sqrt{\frac{t}{2}}\right) e^{-\frac{\delta^2}{2t} - \frac{\alpha^2}{2}t} \mathop{dt}\\
&= 2k! \sum_{j=0}^{\lfloor k/2 \rfloor} \frac{(-1)^j}{j!(k - 2j)!} (-\beta\sqrt{2})^{k - 2j} \left(\frac{\alpha}{\delta}\right)^{j+1} K_{j+1}(\alpha \delta).
\end{align}
After rearranging terms, the series expansion reads as
\begin{equation}\label{general_xmu_small_hermite_series}
F(x; \alpha, \beta, \mu, \delta) = F(\mu; \alpha, \beta, \mu, \delta) + \frac{e^{\delta \gamma} (x-\mu) \alpha}{\pi} \sum_{k=0}^{\infty} \frac{(\beta(x-\mu))^k}{k+1}A_k,
\end{equation}
\begin{equation}
A_k = \sum_{j=0}^{\lfloor k/2 \rfloor} \frac{(-1)^j}{j!(k - 2j)!} \left(\frac{\alpha}{2\delta \beta^2}\right)^j K_{j+1}(\alpha \delta).
\end{equation}

Table \ref{table_xmu_general_n_terms} compares the number of terms of the two presented series expansion for the case $|x-\mu| \to 0$ for various sets of parameters to achieve a machine-precision target accuracy. Note that the Hermite-type expansion \eqref{general_xmu_small_hermite_series} appears to be the most efficient when $|x-\mu| \ll 1$. On the other hand, in the region $|x-\mu| \ge 1$, the expansion \eqref{general_expansion_xmu_small_bessel} becomes the most efficient in the cases considered. 

\begin{table}[htp]
	\centering
	\scalebox{0.9}{
\begin{tabular}{ccccc|cc}
\hline
$x$ & $\alpha$ & $\beta$ & $\mu$ & $\delta$ & \eqref{general_expansion_xmu_small_bessel} 
& \eqref{general_xmu_small_hermite_series}\\
\hline
	1/2 & 2 & 1 & 1/4 & 3 & 27 & 15\\
	1/3 & 5 & -1 & 1/4 & 1 & 12 & 14\\
	4 & 15 & -6 & 7/2 &  10 & 63 & 17\\
	2/10 & 1 & 1/2 & 1/10 & 1/2 & 22 & 19\\
	2 & 3 & 2 & 1 & 2 & 42 & 52\\
	3 & 5 & 1/2 & 1 & 4 & 29 & 66\\
	\hline
\end{tabular}}
	\caption{Comparison on the number of terms to achieve machine precision with the three series expansions for $|x-\mu| \to 0$.}
	\label{table_xmu_general_n_terms}
\end{table}

Note, however, that a complexity analysis gives a different picture. Given a number of terms $N$:
\begin{itemize}
\item Expansion \eqref{general_expansion_xmu_small_bessel}: $(N + 1) (N + 2) \sim \mathcal{O}(N^2)$
\item Expansion \eqref{general_xmu_small_hermite_series}: $\frac{1}{2}\left(\lfloor\frac{N - 1}{2} \rfloor \left(\lfloor\frac{N - 1}{2} \rfloor + 1\right) + \lfloor\frac{N}{2} \rfloor \left(\lfloor\frac{N}{2} \rfloor + 1\right) + 2 (N + 1)\right) \sim \mathcal{O}(N^2)$
\end{itemize}
We remark that, although the series \eqref{general_expansion_xmu_small_bessel} might require fewer terms in some cases, asymptotically, it requires 4 times more terms than \eqref{general_xmu_small_hermite_series}; indeed taking the limit we have
\begin{equation*}
\lim_{N\to\infty} \frac{(N + 1) (N + 2)}{\frac{1}{2}\left(\lfloor\frac{N - 1}{2} \rfloor \left(\lfloor\frac{N - 1}{2} \rfloor + 1\right) + \lfloor\frac{N}{2} \rfloor \left(\lfloor\frac{N}{2} \rfloor + 1\right) + 2 (N + 1)\right)} = 4.
\end{equation*}
Note, however, that the Hermite series involves the evaluation of $F(\mu; \alpha, \beta, \mu, \delta)$, which can potentially increase the total time. These observations shall be taken into consideration when designing the algorithm for the general case in Section \ref{subsection_implementation}.

\subsubsection{Expansion $\beta \to 0$} 
For the case $\beta \to 0$, we follow closely the derivation of the Hermite-type expansion for $|x-\mu| \to 0$ in \eqref{general_xmu_small_hermite_series}, hence some steps are omitted. In this case, the expansion involves the special case $\beta = 0$ studied in detail in Section \ref{section_special_case_beta_0}. Replacing the Hermite-type expansion defined in \eqref{phi_expansion_b_small} in integral \eqref{integral_phi}, we obtain
\begin{equation}
F(x; \alpha, \beta, \mu, \delta) = F(x; \gamma, 0, \mu, \delta) + \frac{\delta e^{\delta \gamma}}{2\pi} \sum_{k=0}^{\infty} \frac{(-1)^k (-\beta)^{k+1}}{(k+1)! 2^{k/2}} \int_0^{\infty} H_k\left(\frac{x-\mu}{\sqrt{2t}}\right) t^{k/2 - 1} e^{-\frac{\omega^2}{2t} - \frac{\gamma^2}{2}t} \mathop{dt}.
\end{equation}
We may apply some algebraic manipulations to obtain an expansion in the same form as \eqref{general_xmu_small_hermite_series}, obtaining
\begin{equation}\label{general_beta_small_hermite_series}
F(x; \alpha, \beta, \mu, \delta) = F(x; \gamma, 0, \mu, \delta) -\frac{\beta \delta e^{\delta \gamma}}{\pi}\sum_{k=0}^{\infty} \frac{(\beta(x-\mu))^k}{k+1}B_k,
\end{equation}
\begin{equation}
B_k = \sum_{j=0}^{\lfloor k/2 \rfloor} \frac{(-1)^j}{j!(k - 2j)!} \left(\frac{\omega}{2\gamma (x-\mu)^2}\right)^j K_{j}(\gamma \omega),
\end{equation}
where we can observe the similarity of the terms in both Hermite-type expansions.

\subsubsection{Expansion $\delta \to \infty$}
Let us start from the integral representation in \eqref{integral_phi} replacing the term $\Phi\left(\frac{x-(\mu + \beta t)}{\sqrt{t}}\right)$ by the asymptotic expansion for $t\to \infty$ in \eqref{phi_expansion_at_inf} in terms of the upper incomplete gamma function. Taking $a = x -\mu$ and $b = -\beta$, we obtain
\begin{equation}
F(x; \alpha, \beta, \mu, \delta) \sim \frac{\delta e^{\delta \gamma}}{2\pi} \sum_{k=0}^{\infty}\frac{(-1)^{k+1}}{2^k k!} \frac{\Gamma(2k+1, -(x-\mu)\beta)}{(-\beta)^{2k+1}} \int_0^{\infty} t^{-k-2} e^{-\frac{\delta^2}{2t} - \frac{\gamma^2 + \beta^2}{2}t} \mathop{dt}.
\end{equation}
The integral is expressible in terms of modified Bessel function
\begin{equation}
\int_0^{\infty} t^{-k-2} e^{-\frac{\delta^2}{2t} - \frac{\gamma^2 + \beta^2}{2}t} \mathop{dt} = 2 \left(\frac{\alpha}{\delta}\right)^{k + 1} K_{k+1}(\alpha \delta),
\end{equation}
where we use the fact that $\gamma^2+ \beta^2 = \alpha^2$. After rearranging terms, we have for $\beta > 0$
\begin{align}\label{general_asymptotic_delta}
F(x; \alpha, \beta, \mu, \delta) &\sim \frac{\alpha e^{\delta \gamma}}{\pi \beta}\sum_{k=0}^{\infty} \frac{(-1)^k}{k!} \Gamma(2k + 1, -(x-\mu)\beta)\left(\frac{\alpha}{2\beta^2 \delta}\right)^k K_{k+1}(\alpha \delta)\nonumber\\
 &=\frac{\alpha e^{\delta \gamma}}{\pi \beta}\sum_{k=0}^{\infty} Q(2k + 1, -(x-\mu)\beta)\frac{\Gamma(2k+1)}{\Gamma(k+1)}\left(-\frac{\alpha}{2\beta^2 \delta}\right)^k K_{k+1}(\alpha \delta)\nonumber\\
 &=\frac{\alpha e^{\delta \gamma}}{\pi^{3/2} \beta}\sum_{k=0}^{\infty} Q(2k + 1, -(x-\mu)\beta) \Gamma(k + 1/2)\left(-\frac{2\alpha}{\beta^2 \delta}\right)^k K_{k+1}(\alpha \delta),
\end{align}
and $1 + F(x; \alpha, \beta, \mu, \delta)$ for $\beta < 0$. We can rewrite the asymptotic series in terms of the regularized incomplete gamma function $Q(a, x) = \Gamma(a, x)/\Gamma(a)$, to avoid ratios with large numerator and denominator. In addition, the evaluation of the upper incomplete gamma function, $\Gamma(a, z)$, and its regularized version via recurrence is numerically stable for positive parameters $a$ and $z$ \cite{Temme1996},
\begin{equation}
\Gamma(a + 1, z) = a \Gamma(a, z) + z^a e^{-z}.
\end{equation}

Furthermore, the special case $x=\mu$ simplifies the asymptotic expansion considerably since $\Gamma(a, 0) = 1$ for $a > 0$, thus obtaining for $\beta > 0$
\begin{equation}\label{expansion_x_eq_mu_large_delta}
F(\mu; \alpha, \beta, \mu, \delta) \sim \frac{\alpha e^{\delta \gamma}}{\pi^{3/2} \beta}\sum_{k=0}^{\infty}\Gamma(k + 1/2)\left(-\frac{2\alpha}{\beta^2 \delta}\right)^k K_{k+1}(\alpha \delta).
\end{equation}

\subsubsection{Expansion $|x-\mu| \to \infty$}
In a similar manner, using instead the expansion in \eqref{phi_expansion_incgamma_t_small}, we obtain for $x-\mu < 0$ the asymptotic expansion
\begin{equation}\label{general_asymptotic_xmu}
F(x; \alpha, \beta, \mu, \delta) \sim -\frac{\delta e^{\delta\gamma}}{\pi (x-\mu)}\sum_{k=0}^{\infty}\frac{(-1)^k}{k!}\Gamma(2k + 1, -(x-\mu)\beta) \left(\frac{\omega}{2\gamma (x-\mu)^2}\right) K_k(\gamma \omega),
\end{equation}
and use $1 + F(x; \alpha, \beta, \mu, \delta)$ for $x-\mu > 0$. The same comments about the computation of the series using the recurrence of the upper incomplete gamma and modified Bessel function apply.

\subsubsection{Uniform asymptotic $\gamma \to \infty, \gamma \sim \delta$ and $|x-\mu| \gg 0$}
For large $\gamma$, we can employ the uniform asymptotic described in Section \ref{uniform_expansion_alpha_large}. The only difference is the calculation of the coefficients $c_k(r)$ using the recurrence in \eqref{phi_expansion_at_u} with $a = x-\mu$ and $b = -\beta$. The resulting uniform asymptotic expansion reads
\begin{equation}
F(x; \alpha, \beta, \mu, \delta) = \frac{\gamma \delta e^{\delta \gamma}}{2\sqrt{\pi}} \sum_{k=0}^{\infty} \frac{2^k c_k\left(\frac{\delta}{\gamma}\right) Q_k\left(\frac{\gamma\delta}{2}\right)}{\gamma^{2k}}, \quad \gamma \to \infty.
\end{equation}

Similar to the comment in Section \ref{uniform_expansion_alpha_large}, the series expansion \eqref{general_beta_small_hermite_series} for $\beta \to 0$ can also be interpreted as an asymptotic for $\gamma \to \infty$, but only admits small values of $|x-\mu|$.

\subsection{Numerical integration}\label{section_numerical_integration}

For cases not handled by the described expansions, we have no other alternative than resorting to numerical integration. However, although numerical integration is an excellent backup method when series expansions are not appropriate, there is an important efficiency aspect to be considered since it might be orders of magnitude slower. 

\subsubsection{Integrand in terms of the error function}
First, we focus on the Laplace-type integral \eqref{integral_phi}, whose integrand involves the complementary error function.
As a preliminary step, we truncate the improper integral \eqref{integral_phi} at some point $N > 0$, such that
\begin{equation}\label{truncated_integral}
I_N = \int_0^N \Phi\left(\frac{x - (\mu +\beta t)}{\sqrt{t}}\right) t^{-3/2} e^{-\frac{\delta^2}{2t} - \frac{\gamma^2}{2}t} \mathop{dt} + \int_N^{\infty} \Phi\left(\frac{x - (\mu +\beta t)}{\sqrt{t}}\right) t^{-3/2} e^{-\frac{\delta^2}{2t} - \frac{\gamma^2}{2}t} \mathop{dt},
\end{equation}
and $F(x; \alpha, \beta, \mu, \delta) = C I_N$, where $C = \frac{\delta e^{\delta \gamma}}{\sqrt{2\pi}}$. The truncation error can be bounded by
\begin{equation}
\int_N^{\infty} \Phi\left(\frac{x - (\mu +\beta t)}{\sqrt{t}}\right) t^{-3/2} e^{-\frac{\delta^2}{2t} - \frac{\gamma^2}{2}t} \mathop{dt} \le \frac{e^{-\frac{\delta^2}{2N}}}{N^{3/2}}\int_N^{\infty} e^{- \frac{\gamma^2}{2}t} \mathop{dt} \le \frac{2 e^{-\frac{\delta^2}{2N} - \frac{\gamma^2}{2}N}}{N^{3/2} \gamma^2},
\end{equation}
and we can select $N$ for a desired absolute tolerance $\epsilon$ solving numerically the equation
\begin{equation}
\frac{2 e^{-\frac{\delta^2}{2N} - \frac{\gamma^2}{2}N}}{N^{3/2} \gamma^2} = \frac{\epsilon}{C}.
\end{equation}
Moreover, a slightly lesser sharper bound allows a closed-form solution of the above equation in terms of the principal branch of the Lambert $W$ function, $W_0(z)$
\begin{equation}\label{N_equation}
\frac{2 e^{- \frac{\gamma^2}{2}N}}{N^{3/2} \gamma^2} = \frac{\epsilon}{C} \longrightarrow N = \frac{3}{\gamma^2}W_0\left(\frac{\gamma^2}{3u}\right), \quad u = \left(\frac{\gamma^2 \epsilon}{2 C}\right)^{2/3}.
\end{equation}
To reduce the cost of evaluating $N$, we can accurately approximate $W_0(z)$ for $z > 1$ using the upper bound derived in \cite{Hoorfar2008}
\begin{equation}\label{lambertW0_upper_bound}
W_0(z) \le \log(z)^{\tfrac{\log(z)}{1 + \log(z)}},
\end{equation}
implying that we can safely use this approximation when the following inequality is satisfied (which should occur for any reasonable set of parameters)
\begin{equation}
\delta \gamma e^{\delta\gamma} > \sqrt{\frac{27 \pi}{2}} \epsilon.
\end{equation}
Note that, from an implementation perspective, the use of the upper bound  \eqref{lambertW0_upper_bound} to estimate $N$ comes in handy to avoid overflow/underflow issues for large parameters by direct use of logarithmic properties. 
To accurately estimate $N$ to achieve a relative tolerance, we need an estimate of the order of magnitude of $I$. The integrand can be written as $e^{g(t)}$, where
\begin{equation}
g(t) = -\frac{\delta^2}{2t} - \frac{\gamma^2}{2}t - \frac{3}{2}\log(t) + \log\left(\Phi\left(\frac{x - (\mu +\beta t)}{\sqrt{t}}\right)\right),
\end{equation}
and its derivative with respect to $t$
\begin{equation}\label{saddle_point_equation}
g'(t) = \frac{\delta^2}{2t^2} -\frac{\gamma^2}{2} -\frac{3}{2t} -\frac{1}{2}\left(\frac{x-\mu}{t^{3/2}} + \frac{\beta}{\sqrt{t}} \right)\frac{\phi\left(\frac{x - (\mu +\beta t)}{\sqrt{t}}\right)}{\Phi\left(\frac{x - (\mu +\beta t)}{\sqrt{t}}\right)}.
\end{equation}
The saddle point $t_0$ and maximum contribution $e^{g(t_0)}$ of the integrand is obtained as the solution of the equation $g'(t) = 0$. Thus, $N$ for relative tolerance can be estimated after replacing $\epsilon$ with $\epsilon e^{g(t_0)}$ in \eqref{N_equation}. For the case where $\gamma$ and $\delta$ are both large and $\beta$ and $x-\mu$ are fixed, the last term in $g'(t)$ can be neglected, obtaining the quadratic equation
\begin{equation}\label{saddle_point_1}
g'(t) \approx \frac{\delta^2}{2t^2} -\frac{\gamma^2}{2} -\frac{3}{2t}, \quad t_0 = \frac{-\frac{3}{2} + \sqrt{\frac{9}{4} + (\gamma \delta)^2}}{\gamma^2},
\end{equation}
taking the positive internal saddle point $t_0$. The case where $\gamma$ and $\delta$ are small requires further analysis. If $x-\mu > 0$, $t_0$ in \eqref{saddle_point_1} is still valid. Contrarily, if $x -\mu < 0$, the ratio of the cumulative distribution function and the density function can be approximated via
\begin{equation*}
\frac{\phi\left(\frac{x - (\mu+ \beta t)}{\sqrt{t}}\right)}{\Phi\left(\frac{x - (\mu + \beta t)}{\sqrt{t}}\right)} \approx -\frac{x - (\mu + \beta t)}{\sqrt{t}},
\end{equation*}
then, replacing in \eqref{saddle_point_equation}, we have another quadratic equation
\begin{equation}\label{saddle_point_2}
g'(t) \approx \frac{\delta^2 + (x-\mu)^2}{2t^2} -\frac{\alpha^2}{2} -\frac{3}{2t}, \quad t_0 = \frac{-\frac{3}{2} + \sqrt{\frac{9}{4} + \alpha^2 \left((x-\mu)^2 + \delta^2\right)}}{\alpha^2}.
\end{equation}
The saddle point estimate $t_0$ computed by \eqref{saddle_point_1} or \eqref{saddle_point_2} is used as a starting point for the root-finding method. In particular, we find that the Newton's method setting an absolute error $10^{-4}$ only requires a few iterations (typically less than 5) to refine the initial estimate. After estimating $N$ using the calculated magnitude with the saddle point $t_0$, $\epsilon e^{g(t_0)}$, the truncated integral \eqref{truncated_integral} can be computed using standard numerical integration methods such as double exponential (tanh-sinh) quadrature or Gaussian quadrature (Gauss-Legendre). Technical details are discussed in Section \ref{algorithmic_numerical_integration}.

\subsubsection{Integrand in terms of elementary functions}
Most of the integral representations presented in Section \ref{properties_cdf} include special functions in their integrands (e.g., $K_1(x)$ or $\erfc(x)$), which necessarily increases the computation cost of their evaluation. Therefore, for computational efficiency, it is generally preferred to evaluate integrals representable in terms of elementary functions. A suitable integral representation satisfying this requirement was given in Equation \eqref{integral_sine_transform}. Following the steps previously described, the integral \eqref{integral_sine_transform} can be approximated by the truncated integral 
\begin{equation}\label{truncated_integral_sine_transform}
I_N = \int_0^N \frac{t e^{-(x-\mu)\left(\sqrt{t^2 + \alpha^2} - \beta\right)}}{\sqrt{t^2 + \alpha^2}\left(\sqrt{t^2 + \alpha^2} - \beta\right)}\sin(\delta t)\mathop{dt} + \int_N^{\infty} \frac{t e^{-(x-\mu)\left(\sqrt{t^2 + \alpha^2} - \beta\right)}}{\sqrt{t^2 + \alpha^2}\left(\sqrt{t^2 + \alpha^2} - \beta\right)}\sin(\delta t)\mathop{dt},
\end{equation}
and $F(x; \alpha, \beta, \mu, \delta) = 1 - C I_N$, where $C = \frac{e^{\delta \gamma}}{\pi}$. Consider the tail given by
\begin{equation}
T_N = \int_N^{\infty} \frac{t e^{-(x-\mu)\left(\sqrt{t^2 + \alpha^2} - \beta\right)}}{\sqrt{t^2 + \alpha^2}\left(\sqrt{t^2 + \alpha^2} - \beta\right)}\sin(\delta t)\mathop{dt}.
\end{equation}
An upper bound for the tail holds:
\begin{align*}
|T_N| &\le \int_N^{\infty} \frac{t e^{-(x-\mu)\left(\sqrt{t^2 + \alpha^2} - \beta\right)}}{\sqrt{t^2 + \alpha^2}\left(\sqrt{t^2 + \alpha^2} - \beta\right)}\mathop{dt}\\
&\le \frac{e^{(x-\mu)\beta}}{\sqrt{N^2 + \alpha^2} - \beta}\int_N^{\infty} e^{-(x-\mu)\sqrt{t^2 + \alpha^2}} \mathop{dt}\\
&\le \frac{e^{(x-\mu)\beta}}{\sqrt{N^2 + \alpha^2} - \beta} \sqrt{N^2 + \alpha^2} K_1((x-\mu)\sqrt{N^2 + \alpha^2}),
\end{align*}
where we use the fact that $|\sin(\delta t)| \le 1$, and a bound for the latter integral expressible in closed form using \cite[\S 3.461]{gradshteyn2007}
\begin{equation*}
\int_N^{\infty} e^{-(x-\mu)\sqrt{t^2 + \alpha^2}} \mathop{dt} \le \int_0^{\infty} e^{-(x-\mu)\sqrt{t^2 + N^2 + \alpha^2}} \mathop{dt} = \sqrt{N^2 + \alpha^2} K_1((x-\mu)\sqrt{N^2 + \alpha^2}).
\end{equation*}

\begin{remark}The previous bound involving the modified Bessel function $K_1(z)$ can be simplified at the expense of obtaining a slightly less sharper bound. Using the monotonicity properties of $K_{\nu}(z)$ for $z > 0$, we have that $K_1(z) < K_{3/2}(z)$, with $K_{3/2}(z)$ being expressible via elementary functions \eqref{besselk_half}
\begin{equation*}
\quad K_{\frac{3}{2}}(z) = \sqrt{\frac{\pi}{2}}\frac{e^{-z} (z+1)}{z^{3/2}}.
\end{equation*}
Hence, one can obtain
\begin{equation}
|T_N| \le  \sqrt{\frac{\pi}{2}}\frac{\sqrt{N^2 + \alpha^2}}{\sqrt{N^2 + \alpha^2} - \beta} \frac{(x-\mu)\sqrt{N^2 + \alpha^2} + 1}{\left((x-\mu) \sqrt{N^2 + \alpha^2}\right)^{3/2}} e^{-(x-\mu)\left(\sqrt{N^2 + \alpha^2} - \beta\right)}.
\end{equation}
\end{remark}
Although the derived bounds are sharp, the inversion to obtain $N$ is far from trivial. In contrast, a practical approximation, improving as $N$ increases $(N \gg \alpha > |\beta|)$ is given by
\begin{equation}
|T_N| \approx \sqrt{\frac{\pi}{2}} \frac{e^{-(x-\mu)N}}{(x-\mu) N}.
\end{equation}
To select $N$, we take the approximation and equate to a prescribed accuracy $\epsilon$. The solution admits a closed-form in terms of $W_0(z)$
\begin{equation}
\frac{e^{-(x-\mu)N}}{(x-\mu)N} = \frac{\epsilon}{C} \longrightarrow N = \frac{1}{x-\mu}W_0\left(\frac{C \sqrt{\frac{\pi}{2}}}{\epsilon}\right).
\end{equation}

It is crucial to mention that the integrand in \eqref{truncated_integral_sine_transform} becomes highly oscillatory as $\delta$ increases. Therefore, the direct computation in the presence of high oscillations leads to catastrophic cancellation, requiring a prohibited number of quadrature points for convergence. In the literature, there is a myriad of numerical quadrature methods for highly oscillatory integrals, for example \cite{Ooura1991}, and we encourage the interested reader to explore them.  Nevertheless,  in this work, we decided solely to investigate the application of the steepest descent method, which we found adequate for this case. We transform the integral with sine kernel into the complex form
\begin{equation}\label{sine_transform_complex_integral}
F(x; \alpha, \beta, \mu, \delta) = 1 - \frac{e^{\delta \gamma}}{\pi}\Im\left(\int_0^{\infty} \frac{t e^{-(x-\mu)\left(\sqrt{t^2 + \alpha^2} - \beta\right) + \delta t i}}{\sqrt{t^2 + \alpha^2}\left(\sqrt{t^2 + \alpha^2} - \beta\right)}\mathop{dt}\right), \quad x-\mu > 0,
\end{equation}
and apply the change of variable $t = i q / \delta$ which yields
\begin{equation}
F(x; \alpha, \beta, \mu, \delta) = 1 + \frac{e^{\delta \gamma + (x-\mu) \beta}}{\pi \delta^2} \Im\left(\int_{0}^{\infty} \frac{q e^{-(x-\mu)\sqrt{\alpha^2 - q^2/\delta^2} - q}}{\sqrt{\alpha^2 - q^2/\delta^2}\left(\sqrt{\alpha^2 - q^2/\delta^2} -\beta\right)} \mathop{dq}\right).
\end{equation}
The integrand has two poles located at $q = \delta \alpha$ and $q = \delta\sqrt{\alpha^2 - \beta^2} = \delta \gamma$, but observe that the imaginary part of the integral only differs from zero when $q > \delta \alpha$. Consequently, the interval of integration can be shifted to $(\delta \alpha, \infty)$ removing the pole ($\delta \alpha \ge \delta \gamma$) which gives
\begin{equation}\label{sine_transform_complex_integral_non_oscillatory}
F(x; \alpha, \beta, \mu, \delta) = 1 + \frac{e^{\delta \gamma + (x-\mu) \beta}}{\pi \delta^2} \Im\left(\int_{\delta\alpha}^{\infty} \frac{q e^{-(x-\mu)i\sqrt{q^2/\delta^2 - \alpha^2} - q}}{i\sqrt{q^2/\delta^2 - \alpha^2}\left(i\sqrt{q^2/\delta^2 - \alpha^2} -\beta\right)} \mathop{dq}\right),
\end{equation}
where the resulting integrand prevents the evaluation of complex square roots. The previously mentioned numerical integration methods are applicable to the non-oscillatory integral \eqref{sine_transform_complex_integral_non_oscillatory}, but also Gauss-Laguerre quadrature with weight function $q^r e^{-q}$, $r \in \{0, 1\}$.

\section{Algorithmic details and implementation}
\subsection{Numerical integration}\label{algorithmic_numerical_integration}
For the computation of the truncated half-line integral in \eqref{truncated_integral}, we choose the double-exponential quadrature method, also known as tanh-sinh quadrature, originally developed in \cite{Takahasi1973}. The double-exponential quadrature has risen in popularity in scientific computing, and compared to the classical Gaussian quadratures, computing the nodes and weights for integration is faster while providing a robust handling of endpoint singularities. Our C++ implementation is inspired by the Python library \texttt{tanh-sinh} \url{https://github.com/sigma-py/tanh-sinh} which carefully implements many of the optimizations described in \cite{Bailey2006TanhSinhHQ}. In particular, it includes an analysis of the determination of the optimal step size, and the number of quadrature points by level, which demands multiple evaluations of the Lambert $W$ function $W_{-1}(z)$. For that purpose, as an optimization, we replace its exact computation by the approximation in \cite{Barry2004}
\begin{equation}
W_{-1}(z) \approx \log(-z) - \frac{2}{a} \left(1 - \left(1 + a \sqrt{-\frac{1 + \log(-z)}{2}}\right)^{-1}\right), \quad	a = 0.3205,
\end{equation}
with maximum relative error of $3.5\cdot 10^{-3}$.

\subsection{Implementation}\label{subsection_implementation}
A reliable implementation of the cumulative distribution function for the normal inverse Gaussian distribution in double-precision arithmetic presents several difficulties. Consequently, it becomes necessary to find the right combination of methods to perform efficient and accurate computations on a wide range of the parameters' domain. To devise this combination (i.e., decide the method of choice in each parameter regime) one could formulate an optimization problem with two possible objective functions: minimization of terms (complexity) and maximization of robustness (accuracy), or a blended approach.

As discussed throughout this work, the achievable accuracy of each method can be effectively estimated by developing rigorous error bounds or, in the absence of error bounds, via linear search with a rough initial estimate. However, it is not always practical to obtain the optimal solution for each set of parameters, and a heuristic might suffice for the vast majority of cases. One can treat the problem of automatically choosing the most suitable method for each region as a classification problem where Machine Learning algorithms are suitable \cite{Simpson2016, Navas-Palencia2019}. More specifically, we can consider a multi-class problem, where each class corresponds to a method, or a multi-label classification problem since multiple methods might be selected for a given set of parameters.
Our approach consists of generating a dataset covering small and large regions of the parameters' domain, computing the cumulative distribution function via \texttt{mpmath} \cite{mpmath} using numerical quadrature, and training a decision tree for each method to decide the regions with absolute relative error below a prescribed tolerance of $5\cdot 10^{-13}$. As a refinement, we combine and simplify the model using another decision tree and posterior rounding of the boundaries. The result of the process are the algorithms 1-3. Some remarks:
\begin{itemize}
\item Numerical integration is used as a backup method whenever series expansions and asymptotic expansions are not suitable or, they fail to converge.
\item For asymptotic expansions, we slightly relaxed tolerance requirements to reduce the use of numerical integration and improve performance.
\end{itemize}

{\centering
\begin{minipage}{.8\linewidth}
  \begin{algorithm}[H]
  \small
  \caption{Algorithm for special case $\beta = 0$, $F(x; \alpha, 0, \mu, \delta)$}\label{alg:case_beta_eq_zero}
\begin{algorithmic}
\Require $x \in \mathbb{R}$, $\alpha > 0$, $\mu \in \mathbb{R}$, $\delta > 0$
\Ensure $F(x; \alpha, 0, \mu, \delta)$
\State $C_1$ = $|x-\mu| \le 5$ and $\alpha / \omega \le 0.25$ and $\delta/2 \ge |x-\mu| $
\State $C_2$ = $(x-\mu)^2 \le 1.25$ and $\alpha / \omega \le 1$
\If{($C_1$ or $C_2$) and $\delta \ge 1$}
    \State series expansion \eqref{expansion_xmu_b0_positive}
\ElsIf{$(x-\mu)^2 \le 2.5$ and $\alpha \ge 5$ and $\delta \ge 10$ and $\delta \alpha \ge 200$}
    \State uniform asymptotic expansion \eqref{uniform_expansion_a_large}
\ElsIf{$(x-\mu)^2 \ge 70$ and $\alpha / \omega \ge 1$}
    \State asymptotic expansion \eqref{asymptotic_expansion_xmu_b0}
\Else
	\State numerical integration \eqref{truncated_integral}
\EndIf
\end{algorithmic}
  \end{algorithm}
\end{minipage}
\par
}

{\centering
\begin{minipage}{.85\linewidth}
  \begin{algorithm}[H]
  \small
  \caption{Algorithm for special case $x=\mu$, $F(\mu; \alpha, \beta, \mu, \delta)$}\label{alg:case_x_eq_mu}
\begin{algorithmic}
\Require $x \in \mathbb{R}$, $\alpha > 0$, $0 \le |\beta| < \alpha$, $\delta > 0$
\Ensure $F(\mu; \alpha, \beta, \mu, \delta)$
\If{$\alpha \le 10$ and $\delta \le 10$ and $|\beta| \le 1.5$ and $|\beta| / \alpha \le 0.9$}
    \State series expansion \eqref{series_x=mu_2}
\ElsIf{$|\beta| / \alpha \ge 0.75$ and $\delta \alpha \ge 300$ and $\delta \ge 15$}
    \State uniform asymptotic expansion \eqref{expansion_x_eq_mu_large_delta}
\Else
	\State numerical integration \eqref{truncated_integral}
\EndIf
\end{algorithmic}
  \end{algorithm}
\end{minipage}
\par
}

{\centering
\begin{minipage}{.85\linewidth}
  \begin{algorithm}[H]
  \small
  \caption{Algorithm for $F(x; \alpha, \beta, \mu, \delta)$}\label{alg:case_general}
\begin{algorithmic}
\Require $x \in \mathbb{R}$, $\alpha > 0$, $0 \le |\beta| < \alpha$, $\mu \in \mathbb{R}$, $\delta > 0$
\Ensure $F(x; \alpha, \beta, \mu, \delta)$
\If{($|\beta| \le 1$ and $\gamma \ge 1.5$) or ($|\beta| \le 0.5$ and $\gamma \ge 0.75$)}
    \State series expansion \eqref{general_beta_small_hermite_series}
\ElsIf{$(x-\mu)^2 \le 2.25$ and $\delta \ge 2.5$}
    \State series expansion \eqref{general_xmu_small_hermite_series}
\ElsIf{$(x-\mu)^2 \le 3$ and $\delta \ge 1$ and $|\beta| \le 1.5$ and $\gamma \ge 0.75$}
	\State series expansion \eqref{general_expansion_xmu_small_bessel}
\ElsIf{$(x-\mu)^2 \le 20$ and $\alpha \ge 5$ and $|\beta| / \alpha \ge 0.5$ and $\delta \ge 15$}
	\State asymptotic expansion \eqref{general_asymptotic_delta}
\ElsIf{$(x-\mu)^2 \ge 100$ and $\alpha / \omega \ge 0.25$ and $\gamma \ge 10$ and $\delta \le 10$ and $\alpha / |\beta| \ge 5$}
	\State asymptotic expansion \eqref{general_asymptotic_xmu}
\Else
	\State numerical integration \eqref{truncated_integral}
\EndIf
\end{algorithmic}
  \end{algorithm}
\end{minipage}
\par
}
\vspace{5mm}
Finally, the resulting algorithm implemented in C++17 is freely available on GitHub: \url{https://github.com/guillermo-navas-palencia/normal-inverse-gaussian}.

\section{Numerical experiments}
The number of available open-source and commercial implementations of the normal inverse Gaussian distribution is limited. In the open-source space, this statistical distribution is not included in widely used numerical libraries such as Boost \cite{BoostLibrary}, GNU Scientific Library (GSL) \cite{GSL} or CERN ROOT \cite{ROOT} to mention a few. To the author's knowledge, the two exceptions are SciPy \texttt{stats.norminvgauss} \cite{scipy} and the R packages \texttt{GeneralizedHyperbolic}\footnote{\url{https://rdrr.io/rforge/GeneralizedHyperbolic/}} and \texttt{ghyp}\footnote{\url{https://rdrr.io/cran/ghyp/}}. The implementation of the cumulative distribution function in these libraries solely relies on the use of numerical integration, more specifically performing adaptive Gaussian quadrature using the Fortran library QUADPACK \cite{QUADPACK}. SciPy performs direct computation via adaptive Gaussian quadrature, whereas \texttt{GeneralizedHyperbolic} partitions the real line into multiple regions, and the numerical integration routine is used in each to integrate the density function.

We compare our implementation to the SciPy \texttt{stats.norminvgauss} since it seems the fastest and most robust of the aforementioned options. To test accuracy, we generate two sample sets, small and large, with $10^4$ instances for each special case and the general case. The small set, which mainly comprises small parameters, was generated to test the accuracy of the series expansions. On the other hand, the large set assesses the accuracy of the asymptotic expansions. For comparing accuracy, we take as a reference the value obtained by computing the truncated integral \eqref{truncated_integral} using \texttt{mpmath} with 100 digits of precision. A run is considered successful if the absolute relative error is below the threshold $5\cdot 10^{-13}$ ($\sim 200 \epsilon$).
Furthermore, we define as hard instances cases where \texttt{mpmath} does not converge and an increment of working precision to 300 digits is required.  
In terms of execution speed, we run the C++ compiled library from Python using \texttt{ctypes} to guarantee a fair comparison. Tests were run on an AMD Ryzen 7 5825U running at 2.0 GHz using the \texttt{clang} compiler with optimization flag \texttt{-O3}.

Following, we outline the main observations for each of the cases:

\paragraph{Case $\beta = 0$}
\begin{itemize}
\item Tables \ref{table_parameters_case_beta_eq_zero} and \ref{table_comparison_case_beta_eq_zero_summary} show the characteristics of the sample sets and the accuracy and performance summary. In terms of accuracy, the SciPy implementation has significantly higher mean absolute error for all the regions, and, therefore, a higher number of failures. In contrast, our implementation shows superior robustness, maintaining a mean error below $10^{-13}$. Regarding performance, our implementation is consistently around $10-20$ times faster than SciPy.
\item Tables \ref{table_methods_case_beta_eq_zero_small} and \ref{table_methods_case_beta_eq_zero_large} show that numerical integration is used in above 60\% of the cases. However, the results in Table \ref{table_methods_case_beta_eq_zero_large_hard} show that for large hard instances, generally corresponding to the case where $|x-\mu|$ is large, the asymptotic expansion \eqref{asymptotic_expansion_xmu_b0} is the preferred choice.
\end{itemize}

\begin{table}[htbp]
\centering
\scalebox{0.9}
{
	\begin{tabular}{c|ccccc}
	\hline
	Region & $x$ & $\alpha$ & $\mu$ & $\delta$\\
	\hline	
	small & (-5, 5) & (0.001, 5) & (-5, 5) & (0.001, 5)\\
	large & (-10, 10) & (0.001, 50) & (-10, 10) & (0.001, 50)\\
	\hline	
	\end{tabular}
}
\caption{Case $\beta = 0$: Parameter range for small and large region.}
	\label{table_parameters_case_beta_eq_zero}
\end{table}

\begin{table}[htbp]
\centering
\scalebox{0.9}
{
\begin{tabular}{ccccc|cc}
\hline
Library & Region & Success & Median & Mean &  Time (s) & Time ($\mu s$)\\
\hline
SciPy & Small & 3952 / 5000 (79.04\%) &  $9.44\cdot 10^{-16}$ & $9.90 \cdot 10^{-4}$ & 2.45 & 490\\
Paper & Small & 4988 / 5000 (99.76\%) & $2.66\cdot 10^{-15}$ & $1.53\cdot 10^{-14}$ & 0.11 & 21\\
	\hline\
SciPy & Large & 3549 / 4868 (72.90\%) & $1.49\cdot 10^{-14}$ & $5.63\cdot 10^{-3}$ & 2.60 & 535\\
Paper & Large & 4868 / 4868 (100\%) & $1.11\cdot 10^{-15}$ & $1.27\cdot 10^{-14}$ & 0.22 & 45\\
	\hline\
SciPy & Large (hard) & 25 / 132 (18.94\%) & $1.44\cdot 10^{-4}$ & $6.45\cdot 10^{-3}$ & 0.028 & 214\\
Paper & Large (hard) & 127 / 132 (96.21\%) & $2.46\cdot 10^{-14}$ & $8.84\cdot 10^{-14}$ & 0.002 & 12\\
	\hline
	\end{tabular}
	}
	\caption{Summary of the accuracy and performance comparison for the case $\beta = 0$.}
	\label{table_comparison_case_beta_eq_zero_summary}
\end{table}

\begin{table}[htbp]
\centering
\scalebox{0.9}
{
\begin{tabular}{ccccccc}
\hline
Method & Count & min & 25\% & Median & 75\% & max\\
\hline
Integration \eqref{truncated_integral} & 4121 (82.42\%) & 0 & $1.22\cdot 10^{-15}$ & $4.15\cdot 10^{-15}$ & $1.44\cdot 10^{-14}$ & $8.40\cdot 10^{-13}$\\
Series \eqref{expansion_xmu_b0_positive} & 879 (17.58\%) & 0 & 0 & $2.22\cdot 10^{-16}$ & $1.11\cdot 10^{-15}$ & $6.58\cdot 10^{-14}$\\
	\hline
\end{tabular}
}
	\caption{Precision metrics of the numerical methods used for computing in the small region for case $\beta = 0$. The errors are the absolute relative errors compared to the reference solutions obtained using mpmath. Percentiles: 25, 50 (median), 75.}
	\label{table_methods_case_beta_eq_zero_small}
\end{table}

\begin{table}[htbp]
\centering
\scalebox{0.9}
{
\begin{tabular}{ccccccc}
\hline
Method & Count & min & 25\% & Median & 75\% & max\\
\hline
Integration \eqref{truncated_integral} & 3137 (64.44\%) & 0 & $3.33\cdot 10^{-16}$ & $8.88\cdot 10^{-16}$ & $4.11\cdot 10^{-15}$ & $4.84\cdot 10^{-13}$\\
Series \eqref{expansion_xmu_b0_positive} & 810 (16.64\%) & 0 & $3.33\cdot 10^{-16}$ & $1.33\cdot 10^{-15}$ & $8.41\cdot 10^{-15}$ & $4.28\cdot 10^{-13}$\\
Asymptotic \eqref{asymptotic_expansion_xmu_b0} & 620 (12.74\%) & 0 & 0 & 0 & $2.54\cdot 10^{-14}$ & $2.31\cdot 10^{-13}$\\
Asymptotic \eqref{uniform_expansion_a_large} & 301 (6.18\%) & $2.22\cdot 10^{-16}$ & $4.44\cdot 10^{-15}$ & $1.24\cdot 10^{-14}$ & $3.46\cdot 10^{-14}$ & $2.26\cdot 10^{-13}$\\
	\hline
	\end{tabular}
	}
	\caption{Precision metrics of the numerical methods used for computing in the large region for case $\beta = 0$. The errors are the absolute relative errors compared to the reference solutions obtained using mpmath. Percentiles: 25, 50 (median), 75.}
	\label{table_methods_case_beta_eq_zero_large}
\end{table}

\begin{table}[htbp]
\centering
\scalebox{0.9}
{
\begin{tabular}{ccccccc}
\hline
Method & Count & min & 25\% & Median & 75\% & max\\
\hline
Integration \eqref{truncated_integral} & 22 (16.67\%) & 0 & $6.94\cdot 10^{-16}$ & $3.66\cdot 10^{-15}$ & $1.76\cdot 10^{-15}$ & $5.98\cdot 10^{-14}$\\
Series \eqref{expansion_xmu_b0_positive} & 7 (5.30\%) & $4.44\cdot 10^{-16}$ & $6.11\cdot 10^{-16}$ & $2.00\cdot 10^{-15}$ & $5.39\cdot 10^{-13}$ & $3.22\cdot 10^{-12}$\\
Asymptotic \eqref{asymptotic_expansion_xmu_b0} & 100 (75.76\%) & $8.88\cdot 10^{-16}$ & $1.35\cdot 10^{-14}$ & $3.00\cdot 10^{-14}$ & $5.57\cdot 10^{-14}$ & $1.30\cdot 10^{-13}$\\
Asymptotic \eqref{uniform_expansion_a_large} & 3 (2.27\%) & $5.76\cdot 10^{-13}$ & $8.84\cdot 10^{-13}$ & $1.19\cdot 10^{-12}$ & $1.42\cdot 10^{-12}$ & $1.64\cdot 10^{-12}$\\
	\hline
\end{tabular}
}
	\caption{Precision metrics of the numerical methods used for computing in the large (hard) region for case $\beta = 0$. The errors are the absolute relative errors compared to the reference solutions obtained using mpmath. Percentiles: 25, 50 (median), 75.}
	\label{table_methods_case_beta_eq_zero_large_hard}
\end{table}

\newpage
\paragraph{Case $x = \mu$}
\begin{itemize}
\item Tables \ref{table_parameters_case_x_eq_mu} and \ref{table_comparison_case_x_eq_mu_summary} show the characteristics of the sample sets and the accuracy and performance summary. The SciPy implementation shows more reliability in the small region, but its accuracy worsens in the large region, especially for hard instances, where just 1.77\% of these are computed successfully.
\item In the small region, contrasting to the case $\beta = 0$, Table \ref{table_methods_case_x_eq_mu_small} shows that the series expansion \eqref{series_x=mu_2} can be safely chosen in more than $80\%$ of the instances hence the speedup compared to SciPy increases, being $\sim 60$ times faster.
\item Tables \ref{table_methods_case_x_eq_mu_large} and \ref{table_methods_case_x_eq_mu_large_hard} shares the same conclusions than the case $\beta = 0$, numerical integration dominates in the large region, and for hard instances (large parameters $\alpha$ and/or $\delta$), the asymptotic expansion \eqref{expansion_x_eq_mu_large_delta} is effective.
\end{itemize}

\begin{table}[htbp]
\centering
\scalebox{0.9}
{
	\begin{tabular}{c|cccc}
	\hline
	Region & $\alpha$ & $\beta$ & $\delta$\\
	\hline	
	small & (0.001, 5) & (-5, 5) & (0.001, 5)\\
	large & (0.001, 50) & (-50, 50) & (0.001, 50)\\
	\hline
	\end{tabular}
}
\caption{Case $x = \mu$: Parameter range for small and large region.}
	\label{table_parameters_case_x_eq_mu}
\end{table}

\begin{table}[htbp]
\centering
\scalebox{0.9}
{
\begin{tabular}{ccccc|cc}
\hline
Library & Region & Success & Median & Mean & Time (s) & Time ($\mu s$)\\
\hline
SciPy & Small & 4838 / 5000 (96.76\%) &  $8.88\cdot 10^{-16}$ & $8.40 \cdot 10^{-11}$ & 2.02 & 405\\
Paper & Small & 5000 / 5000 (100\%) & $2.22\cdot 10^{-16}$ & $9.60\cdot 10^{-16}$ & 0.03 & 7\\
	\hline\
SciPy & Large & 3084 / 4548 (67.81\%) & $1.96\cdot 10^{-15}$ & $1.34\cdot 10^{-12}$ & 2.58 & 516\\
Paper & Large & 4548 / 4548 (100\%) & $6.66\cdot 10^{-16}$ & $8.25\cdot 10^{-15}$ & 0.27 & 53\\
	\hline\
SciPy & Large (hard) & 8 / 452 (1.77\%) & $1.26\cdot 10^{-1}$ & $1.65\cdot 10^{-1}$ & 0.104 & 229\\
Paper & Large (hard) & 425 / 452 (94.03\%) & $3.86\cdot 10^{-14}$ & $2.47\cdot 10^{-12}$ & 0.024 & 53\\
	\hline
\end{tabular}
}
	\caption{Summary of the accuracy and performance comparison for the case $x=\mu$.}
	\label{table_comparison_case_x_eq_mu_summary}
\end{table}

\begin{table}[htbp]
\centering
\scalebox{0.9}
{
\begin{tabular}{ccccccc}
\hline
Method & Count & min & 25\% & Median & 75\% & max\\
\hline
Integration \eqref{truncated_integral} & 949 (18.48\%) & 0 & 0 & $2.22\cdot 10^{-16}$ & $4.44\cdot 10^{-16}$ & $6.00\cdot 10^{-15}$\\
Series \eqref{series_x=mu_2} & 4051 (81.02\%) & 0 & 0 & $2.22\cdot 10^{-16}$ & $5.55\cdot 10^{-16}$ & $2.48\cdot 10^{-13}$\\
	\hline
\end{tabular}}
	\caption{Precision metrics of the numerical methods used for computing in the small region for case $x=\mu$. The errors are the absolute relative errors compared to the reference solutions obtained using mpmath. Percentiles: 25, 50 (median), 75.}
	\label{table_methods_case_x_eq_mu_small}
\end{table}

\begin{table}[htbp]
\centering
\scalebox{0.9}
{
\begin{tabular}{ccccccc}
\hline
Method & Count & min & 25\% & Median & 75\% & max\\
\hline
Integration \eqref{truncated_integral} & 4289 (94.31\%) & 0 & $2.22\cdot 10^{-16}$ & $5.55\cdot 10^{-16}$ & $3.77\cdot 10^{-15}$ & $4.81\cdot 10^{-13}$\\
Series \eqref{series_x=mu_2} & 126 (2.77\%) & 0 & $2.22\cdot 10^{-16}$ & $4.44\cdot 10^{-16}$ & $1.67\cdot 10^{-15}$ & $2.60\cdot 10^{-13}$\\
Asymptotic \eqref{expansion_x_eq_mu_large_delta} & 133 (2.92\%) & $5.55\cdot 10^{-16}$ & $1.70\cdot 10^{-14}$ & $2.70\cdot 10^{-14}$ & $4.30\cdot 10^{-14}$ & $1.15\cdot 10^{-13}$\\
	\hline
\end{tabular}}
	\caption{Precision metrics of the numerical methods used for computing in the large region for case $x=\mu$. The errors are the absolute relative errors compared to the reference solutions obtained using mpmath. Percentiles: 25, 50 (median), 75.}
	\label{table_methods_case_x_eq_mu_large}
\end{table}

\begin{table}[htbp]
\centering
\scalebox{0.9}
{
\begin{tabular}{ccccccc}
\hline
Method & Count & min & 25\% & Median & 75\% & max\\
\hline
Integration \eqref{truncated_integral} & 204 (45.13\%) & $2.22\cdot 10^{-16}$ & $1.73\cdot 10^{-14}$ & $3.22\cdot 10^{-14}$ & $4.73\cdot 10^{-14}$ & $1.19\cdot 10^{-13}$\\
Series \eqref{series_x=mu_2} & 1 (0.22\%) & $7.76\cdot 10^{-13}$ & $7.76\cdot 10^{-13}$ & $7.76\cdot 10^{-13}$ & $7.76\cdot 10^{-13}$ & $7.76\cdot 10^{-13}$\\
Asymptotic \eqref{expansion_x_eq_mu_large_delta} & 247 (54.65\%) & 0 & $3.09\cdot 10^{-14}$ & $5.43\cdot 10^{-14}$ & $9.60\cdot 10^{-14}$ & $-$\\
	\hline
\end{tabular}}
	\caption{Precision metrics of the numerical methods used for computing in the large (hard) region for case $x=\mu$. The errors are the absolute relative errors compared to the reference solutions obtained using mpmath. Percentiles: 25, 50 (median), 75.}
	\label{table_methods_case_x_eq_mu_large_hard}
\end{table}

\paragraph{General case}
\begin{itemize}
\item Tables \ref{table_parameters_case_general} and \ref{table_comparison_case_general_sumamry} show the characteristics of the sample sets and the accuracy and performance summary. The previous observations largely apply, except for the fact, that the accuracy for large (hard) instances is remarkably lower. The symbol $*$ indicates that \texttt{mpmath} struggles to converge with 300 digits after various attempts varying the integral upper bound $N$, returning systematically a value of zero. Nevertheless, the proposed algorithms seems less reliable for large values of $|x-\mu|$ for the general case.
\item Table \ref{table_methods_case_general_small} and  \ref{table_methods_case_general_large} exhibit that although the use of numerical integration is also dominant, the various expansions achieve similar accuracy. Nonetheless, Table \ref{table_methods_case_general_large_hard} illustrates that for some large hard instances, some expansions suffer to obtain more than 11 correct digits, forcing a switch to numerical integration to improve reliability in these cases.
\item Table \ref{table_comparison_case_general_integration} compares the SciPy adaptive quadrature implementation with our double-exponential quadrature implementation, the latter demonstrating superior accuracy and performance, being $5-20$ times faster. Nevertheless, we emphasize that, although an implementation uniquely based on numerical integration is possible, the use of various expansions can reliably complement with benefits in terms of CPU times (as previously stated, numerical integration tends to be up to 100 times slower than an efficient implementation of expansions exploiting recurrence relations).
\end{itemize}

\begin{table}[htbp]
\centering
\scalebox{0.9}
{
	\begin{tabular}{c|cccccc}
	\hline
	Region & $x$ & $\alpha$ & $\beta$ & $\mu$ & $\delta$\\
	\hline	
	small & (-5, 5) & (0.001, 5) & (-5, 5) & (-5, 5) & (0.001, 5)\\
	large & (-10, 10) & (0.001, 50) & (-50, 50) & (-10, 10) & (0.001, 50)\\
	\hline
	\end{tabular}
}
\caption{General case: Parameter range for small and large region.}
	\label{table_parameters_case_general}
\end{table}

\begin{table}[htbp]
\centering
\scalebox{0.9}
{
\begin{tabular}{ccccc|cc}
\hline
Library & Region & Success & Median & Mean & Time (s) & Time ($\mu s$)\\
\hline
SciPy & Small & 3983 / 5000 (79.66\%) &  $1.78\cdot 10^{-15}$ & $1.89 \cdot 10^{-5}$ & 2.53 & 506\\
Paper & Small & 4939 / 5000 (98.78\%) & $3.89\cdot 10^{-15}$ & $1.01\cdot 10^{-13}$ & 0.12 & 24\\
	\hline\
SciPy & Large & 2949 / 4296 (68.65\%) & $4.49\cdot 10^{-14}$ & $7.91\cdot 10^{-3}$ & 2.79 & 559\\
Paper & Large & 4296 / 4296 (100\%) & $6.66\cdot 10^{-16}$ & $1.09\cdot 10^{-14}$ & 0.44 &88\\
	\hline\
SciPy & Large (hard) & 70 / 704 (9.94\%) & $0.16\cdot 10^{-1}$ & $*$ & 0.16 & 223\\
Paper & Large (hard) & 409 / 704 (58.10\%) & $8.07\cdot 10^{-14}$ & $*$ & 0.07 & 95\\
	\hline
\end{tabular}
}
	\caption{Summary of the accuracy and performance comparison for the general case. Symbol $*$ indicates numerical issues were encountered.}
	\label{table_comparison_case_general_sumamry}
\end{table}

\begin{table}[htbp]
\centering
\scalebox{0.9}
{
\begin{tabular}{ccccccc}
\hline
Method & Count & min & 25\% & Median & 75\% & max\\
\hline
Integration \eqref{truncated_integral} & 3585 (71.70\%) & 0 & $6.66\cdot 10^{-16}$ & $5.00\cdot 10^{-15}$ & $1.78\cdot 10^{-14}$ & $7.56\cdot 10^{-11}$\\
Series \eqref{general_expansion_xmu_small_bessel} & 212 (4.24\%) & 0 & $4.44\cdot 10^{-16}$ & $1.61\cdot 10^{-15}$ & $1.14\cdot 10^{-14}$ & $2.68\cdot 10^{-12}$\\
Series \eqref{general_xmu_small_hermite_series} & 606 (12.12\%) & 0 & $3.33\cdot 10^{-16}$ & $8.88\cdot 10^{-16}$ & $5.30\cdot 10^{-15}$ & $3.08\cdot 10^{-11}$\\
Series \eqref{general_beta_small_hermite_series} & 597 (11.94\%) & 0 & $8.88\cdot 10^{-16}$ & $4.22\cdot 10^{-15}$ & $1.47\cdot 10^{-14}$ & $1.70\cdot 10^{-12}$\\
	\hline
\end{tabular}
}
	\caption{Precision metrics of the numerical methods used for computing in the small region for general case. The errors are the absolute relative errors compared to the reference solutions obtained using mpmath. Percentiles: 25, 50 (median), 75.}
	\label{table_methods_case_general_small}
\end{table}

\begin{table}[htbp]
\centering
\scalebox{0.9}
{
\begin{tabular}{ccccccc}
\hline
Method & Count & min & 25\% & Median & 75\% & max\\
\hline
Integration \eqref{truncated_integral} & 3344 (77.84\%) & 0 & $2.22\cdot 10^{-16}$ & $7.77\cdot 10^{-16}$ & $5.19\cdot 10^{-15}$ & $4.54\cdot 10^{-13}$\\
Series \eqref{general_expansion_xmu_small_bessel} & 9 (0.21\%) & $2.22\cdot 10^{-16}$ & $1.11\cdot 10^{-15}$ & $2.00\cdot 10^{-15}$ & $1.09\cdot 10^{-14}$ & $1.46\cdot 10^{-13}$\\
Series \eqref{general_xmu_small_hermite_series} & 516 (12.01\%) & 0 & $2.22\cdot 10^{-16}$ & $6.66\cdot 10^{-16}$ & $8.83\cdot 10^{-15}$ & $4.90\cdot 10^{-13}$\\
Series \eqref{general_beta_small_hermite_series} & 156 (3.63\%) & 0 & $4.44\cdot 10^{-16}$ & $2.33\cdot 10^{-15}$ & $1.87\cdot 10^{-14}$ & $4.71\cdot 10^{-13}$\\
Asymptotic \eqref{general_asymptotic_xmu} & 13 (0.30\%) &  0 & 0 & 0 & 0 & $2.11\cdot 10^{-14}$\\
Asymptotic \eqref{general_asymptotic_delta} & 258 (6.01\%) &  0 & 0 & 0 & 0 & $3.62\cdot 10^{-13}$\\
	\hline
\end{tabular}
}
	\caption{Precision metrics of the numerical methods used for computing in the large region for general case
	. The errors are the absolute relative errors compared to the reference solutions obtained using mpmath. Percentiles: 25, 50 (median), 75.}
	\label{table_methods_case_general_large}
\end{table}

\begin{table}[htbp]
\centering
\scalebox{0.9}
{
\begin{tabular}{ccccccc}
\hline
Method & Count & min & 25\% & Median & 75\% & max\\
\hline
Integration \eqref{truncated_integral} & 400 (56.82\%) & $2.22\cdot 10^{-15}$ & $2.75\cdot 10^{-14}$ & $4.76\cdot 10^{-14}$ & $7.70\cdot 10^{-14}$ & $-$ \\
Series \eqref{general_expansion_xmu_small_bessel} & 2 (0.29\%) & $9.32\cdot 10^{-13}$ & $1.67\cdot 10^{-11}$ & $3.24\cdot 10^{-11}$ & $4.81\cdot 10^{-11}$ & $6.81\cdot 10^{-11}$\\
Series \eqref{general_xmu_small_hermite_series} & 155 (22.02\%) & $9.99\cdot 10^{-16}$ & $1.11\cdot 10^{-12}$ & $3.75\cdot 10^{-10}$ & $-$ & $-$\\
Series \eqref{general_beta_small_hermite_series} & 13 (1.85\%) & $2.55\cdot 10^{-13}$ & $7.09\cdot 10^{-13}$ & $1.62\cdot 10^{-12}$ & $7.45\cdot 10^{-12}$ & $7.96\cdot 10^{-11}$\\
Asymptotic \eqref{general_asymptotic_xmu} & 14 (1.99\%) & $2.66\cdot 10^{-15}$ & $5.22\cdot 10^{-14}$ & $3.65\cdot 10^{-12}$ & $2.08\cdot 10^{-10}$ & $-$\\
Asymptotic \eqref{general_asymptotic_delta} & 120 (17.05\%) & $2.11\cdot 10^{-15}$ & $3.87\cdot 10^{-13}$ & $1.02\cdot 10^{-9}$ & $-$ & $-$\\
	\hline
\end{tabular}
}
	\caption{Precision metrics of the numerical methods used for computing in the large (hard) region for general case. The errors are the absolute relative errors compared to the reference solutions obtained using mpmath. Percentiles: 25, 50 (median), 75.}
	\label{table_methods_case_general_large_hard}
\end{table}

\begin{table}[htbp]
\centering
\scalebox{0.9}
{
\begin{tabular}{ccccc|cc}
\hline
Library & Region & Success & Median & Mean & Time (s) & Time ($\mu s$)\\
\hline
SciPy & Small & 3983 / 5000 (79.66\%) &  $1.78\cdot 10^{-15}$ & $1.89 \cdot 10^{-5}$ & 2.53 & 506\\
Paper & Small & 4980 / 5000 (99.60\%) & $3.99\cdot 10^{-15}$ & $1.64\cdot 10^{-14}$ & 0.12 & 27\\
	\hline\
SciPy & Large & 3139 / 5000 (62.78\%) & $0.16\cdot 10^{-1}$ & $*$ & 2.74 & 550\\
Paper & Large & 4964 / 5000 (99.28\%) & $8.07\cdot 10^{-14}$ & $*$ & 0.51 & 102\\
	\hline
\end{tabular}}
	\caption{Summary of the accuracy and performance comparison for the general case only using numerical integration. Symbol $*$ indicates numerical issues were encountered.}
	\label{table_comparison_case_general_integration}
\end{table}

\section{Conclusions}
In the work presented, we developed and implemented an algorithm for fast and robust computation of the normal inverse Gaussian cumulative distribution function.
With a thorough treatment, we showed the effectiveness of combining multiple series and asymptotic expansions with numerical integration to obtain significant
acceleration compared to current state-of-the-art open-source implementations while keeping accuracy to near machine precision.

Several potential improvements remain to be explored in future work. 
First, the parameter regime of the convergent series expansions could be widened if computed increasing the working precision to long double (80-bit) or quadruple precision (128 bit), taking into account that quadruple precision is not natively supported in hardware, and necessarily it must be simulated\footnote{For example, using libquadmath, the GCC Quad-Precision Math Library.}, which in turn produces a code at least 10x slower. An alternative is using double-double precision (approximately 30 digits), an approach implemented in libraries such as QD library \cite{Hida2001}, which has been successfully implemented for the computation of special functions, especially near the zeros \cite{Navas-Palencia2018a}. 

Second, it would also be interesting to evaluate the use of fixed-order Gauss-Legendre quadrature instead of the double-exponential quadrature, studying the appropriate number of quadrature points for different regions of the parameters' domain. In addition, it might be worth investigating the implementation of the integral expressible in terms of elementary function for small $\delta$ values. 

Lastly, the techniques introduced throughout this work for developing many of the expansions can, of course, be seamlessly extendible to develop efficient algorithms for the computation of the generalized hyperbolic distribution and other special cases, such as the variance gamma distribution.

\begin{appendices}

\section{The function $\Phi\left(\frac{a}{\sqrt{t}} + b\sqrt{t}\right)$}

The function $F(t; a, b) = \Phi\left(\frac{a}{\sqrt{t}} + b\sqrt{t}\right)$ appears in the integrand of the integral representation in \eqref{integral_phi}. Given its relevance throughout this work, we introduce here some results that shall be used subsequently. $F(t; a, b)$ has the following integral representation \cite[\S 7.7.6]{NIST:DLMF}
\begin{equation}\label{integral_erfc_ab}
F(t; a, b) = \frac{1}{2}\erfc\left(-\frac{\frac{a}{\sqrt{t}} + b\sqrt{t}}{\sqrt{2}}\right)  = \sqrt{\frac{t}{\pi}} e^{-\frac{a^2}{2t}} \int_{-b/\sqrt{2}}^{\infty} e^{-(tu^2 - \sqrt{2}au)} \mathop{du}
\end{equation}

\subsection{Expansions $t$}

\subsubsection{Expansion $t \to 0$}

Let us consider the case $a < 0$, since we can use the mirror property $\Phi(z) = 1 - \Phi(-z)$ otherwise. To obtain an expansion for $t \to 0$, we expand $e^{-tu^2}$ and interchange summation and integration obtaining
\begin{equation*}
F(t; a, b) = \sqrt{\frac{t}{\pi}} e^{-\frac{a^2}{2t}} \sum_{k=0}^{\infty} \frac{(-t)^k}{k!}\int_{-b/\sqrt{2}}^{\infty} e^{\sqrt{2}a u} u^{2k}\mathop{du}.
\end{equation*}
For $a < 0$ the integral can be expressed in closed form in terms of the incomplete gamma function, $\Gamma(a, x)$
\begin{equation*}
\int_{-b/\sqrt{2}}^{\infty} e^{\sqrt{2}a u} u^{2k}\mathop{du} = \frac{\Gamma(2k+1, -ab)}{(\sqrt{2}a)^{2k+1}},
\end{equation*}
and for the special case $b=0$, it reduces to
\begin{equation*}
\int_{0}^{\infty} e^{\sqrt{2}a u} u^{2k}\mathop{du} = \frac{\Gamma(2k+1)}{(\sqrt{2}a)^{2k+1}}.
\end{equation*}
Then, we obtain the series expansion valid for $t \to 0$, $a \to -\infty$ and fixed $b$
\begin{equation}\label{phi_expansion_incgamma_t_small}
F(t; a, b) = \sqrt{\frac{t}{\pi}} e^{-\frac{a^2}{2t}} \sum_{k=0}^{\infty} \frac{(-1)^{k+1} t^k}{k!}\frac{\Gamma(2k + 1, ab)}{(\sqrt{2}a)^{2k+1}}.
\end{equation}

\subsubsection{Expansion $t \to \infty$}

Let us focus on the case $t \to \infty$. We can develop an asymptotic expansion after expanding the term $e^{\sqrt{2}au}$ in \eqref{integral_erfc_ab}, which yields
\begin{equation*}
F(t; a, b) = \sqrt{\frac{t}{\pi}} e^{-\frac{a^2}{2t}} \sum_{k=0}^{\infty}\frac{(\sqrt{2}a)^k}{k!}\int_{-b/\sqrt{2}}^{\infty} e^{-t u^2} u^k \mathop{du}.
\end{equation*}
Considering the case $b < 0$ (again, we can use the mirror property), the integral has a closed-form
\begin{equation*}
\int_{-b/\sqrt{2}}^{\infty} e^{-t u^2} u^k \mathop{du} = \frac{\Gamma\left(\frac{k+1}{2}, \frac{b^2}{2}t\right)}{2 t^{\frac{k+1}{2}}}.
\end{equation*}
Thus,
\begin{equation}\label{phi_expansion_incgamma}
F(t; a, b) = \sqrt{\frac{t}{\pi}} \frac{e^{-\frac{a^2}{2t}}}{2}  \sum_{k=0}^{\infty}\frac{(\sqrt{2}a)^k}{k!} \frac{\Gamma\left(\frac{k+1}{2}, \frac{b^2}{2}t\right)}{t^{\frac{k+1}{2}}}.
\end{equation}
The asymptotic behaviour of the terms in the series is
\begin{equation*}
\frac{\Gamma\left(\frac{k+1}{2}, \frac{b^2}{2}t\right)}{t^{\frac{k+1}{2}}} \sim \left(\frac{b^2}{2}\right)^{\frac{k+1}{2}} e^{-\frac{b^2}{2} t}, \quad t\to\infty.
\end{equation*}
In fact this series is convergent, as can be observed taking the asymptotic estimate of $\Gamma(k, x)$ as $k \to \infty$. A simpler convergent expansion can be obtained transforming the integral in \eqref{integral_erfc_ab}
\begin{equation*}
\sqrt{\frac{t}{\pi}} e^{-\frac{a^2}{2t}} \int_{-b/\sqrt{2}}^{\infty} e^{-(tu^2 - \sqrt{2}au)} \mathop{du} = \sqrt{\frac{t}{\pi}} e^{-\frac{a^2}{2t} -ab - \frac{b^2}{2}t}\int_0^{\infty}e^{\sqrt{2}(a+bt)u} e^{-tu^2}\mathop{dt},
\end{equation*}
and expanding $e^{\sqrt{2}(a+bt)u}$ obtaining
\begin{equation}
F(t; a, b) = \sqrt{\frac{t}{\pi}} \frac{e^{-\frac{a^2}{2t} -ab - \frac{b^2}{2}t}}{2} \sum_{k=0}^{\infty} \frac{(\sqrt{2}(a+bt))^k}{k!}\frac{\Gamma\left(\frac{k+1}{2}\right)}{t^{\frac{k+1}{2}}}.
\end{equation}

In addition, we can obtain an asymptotic expansion directly expanding $F(t; a, b)$ at $t \to \infty$. For $b > 0$, the first terms of the expansion are
\begin{equation}
c_0 = 1, \quad c_1 = 2 + 2ab + a^2 b^2, \quad c_2 = 24 + 24ab + 12a^2b^2 + 4a^3b^3 + a^4b^4,
\end{equation}
\begin{equation}
F(t; a, b) = 1 + \frac{e^{-ab - \frac{b^2}{2}t}}{\sqrt{2\pi}}\sum_{k=0}^{\infty}\frac{(-1)^{k+1}}{2^k k!}\frac{c_k}{b^{2k+1}}\left(\frac{1}{t}\right)^{k+\frac{1}{2}}.
\end{equation}
The coefficients $c_k$ are expressible in terms of the incomplete gamma function, since
\begin{equation*}
c_k = \sum_{j=0}^{2k}\frac{(2k)!}{j!}(ab)^j = e^{ab}\Gamma(2k+1, ab).
\end{equation*}
Rearranging terms, we get
\begin{equation}\label{phi_expansion_at_inf}
F(t; a, b) = 1 + \frac{e^{- \frac{b^2}{2}t}}{\sqrt{2\pi}}\sum_{k=0}^{\infty}\frac{(-1)^{k+1}}{2^k k!}\frac{\Gamma(2k+1, ab)}{b^{2k+1}}\left(\frac{1}{t}\right)^{k+\frac{1}{2}}.
\end{equation}

\subsubsection{Expansion $t \to u$}

Lastly, we study the expansion of $F(t;a,b)$ at $t=u$. This expansion shall be crucial when developing various uniform asymptotic expansions later on. The first coefficients of the Taylor series are
\begin{equation*}
c_0 = \Phi\left(\frac{a}{\sqrt{u}} + b\sqrt{u}\right)d_0,\quad c_1 = \phi\left(\frac{a}{\sqrt{u}} + b\sqrt{u}\right) d_1, \quad c_2 = -\phi\left(\frac{a}{\sqrt{u}} + b\sqrt{u}\right) d_2, \quad
c_3 = \phi\left(\frac{a}{\sqrt{u}} + b\sqrt{u}\right) d_3,
\end{equation*}
where
\begin{align*}
d_0 &= 1\\
d_1 &= \frac{-a + bu}{2u^{3/2}}\\
d_2 &=\frac{a^3 -3au -a^2 bu + bu^2 -ab^2 u^2 + b^3u^3}{8 u^{7/2}}\\
d_3 &= \frac{-a^5 +10a^3u + a^4bu -15au^2 - 6a^2bu^2 + 2a^3 b^2 u^2 +3bu^3 -6ab^2u^3 - 2a^2b^3 u^3 + 2b^3 u^4 -ab^4u^4 + b^5 u^5}{48 u^{11/2}},
\end{align*}
and $\phi(x) = \frac{e^{-x^2/2}}{\sqrt{2\pi}}$ is the probability density function of the standard normal distribution. Thus, we have
\begin{equation}\label{phi_expansion_at_u}
F(t; a, b) = \sum_{k=0}^{\infty} c_k (t-u)^k.
\end{equation}
Additional terms satisfy the following recurrence
\begin{equation}
c_{k+4} = \frac{f_0(k) c_k + f_1(k) c_{k+1} + f_2(k) c_{k+2} + f_3(k) c_{k+3}}{f_4(k)}, \quad k \ge 0
\end{equation}
where
\begin{align*}
f_0(k) &= -kb^3\\
f_1(k) &= -(1 + k) b (1 + 2k -ab + 3b^3u)\\
f_2(k) &= (2 + k)(5a + 2ka +a^2b - 8bu - 6kbu +2ab^2u -3b^3u^2)\\
f_3(k) &= -(3 + k)(a^3 - 11au -4kau - a^2bu + 13bu^2 + 6kbu^2 -ab^2u^2 + b^3 u^3)\\
f_4(k) &= 2 (3 + k) (4 + k) u^2 (-a + bu)
\end{align*}

\subsection{Expansion $a \to 0$}
We also consider expansions for small values of the parameters. If we expand $F(t; a, b)$ at $a = 0$, we obtain the series
\begin{equation}
F(t; a, b) = \Phi\left(b\sqrt{t}\right) + \frac{e^{-\frac{b^2 t}{2}}}{\sqrt{2\pi t}}\sum_{k=1}^{\infty} (-1)^{k+1} \frac{a^k P_{k-1}(t; b)}{k!}.
\end{equation}
The first coefficients $P_k(t; b)$ are
\begin{align*}
P_0(t; b) &= 1\\
P_1(t; b) &= b\\
P_2(t; b) &= \frac{b^2 t - 1}{t}\\
P_3(t; b) &= \frac{b^3 t - 3b}{t}\\
P_4(t; b) &= \frac{b^4 t^2 - 6b^2t + 3}{t^2}\\
P_5(t; b) &= \frac{b^5 t^2 - 10b^3 t + 15 b}{t^2}
\end{align*}
Looking at the first coefficients, it is not hard to observe that coefficients $P_k(t;b)$ are expressible in terms of probabilist Hermite polynomials. Thus, after using the connection formula between probabilist and classical Hermite polynomials, we obtain the following convergent series expansion
\begin{equation}\label{phi_expansion_a_small}
F(t; a, b) = \Phi\left(b\sqrt{t}\right) + \frac{e^{-\frac{b^2 t}{2}}}{\sqrt{2\pi t}} \sum_{k=0}^{\infty} \frac{(-1)^k a^{k+1}}{(k+1)!}\frac{1}{(2t)^{k/2}} H_k\left( b\sqrt{\frac{t}{2}}\right)
\end{equation}

\subsection{Expansion $b \to 0$}
The Hermite-type expansion for the case $b \to 0$ is obtained analogously after expanding $F(t; a, b)$ at $b=0$. Thus,
\begin{equation}\label{phi_expansion_b_small}
F(t; a, b) = \Phi\left(\frac{a}{\sqrt{t}}\right) + e^{-\frac{a^2 t}{2}}\sqrt{\frac{t}{2\pi}} \sum_{k=0}^{\infty}\frac{(-1)^k b^{k+1}}{(k + 1)!}\left(\frac{t}{2}\right)^{k/2} H_k\left(\frac{a}{\sqrt{2t}}\right).
\end{equation}

\end{appendices}

\end{document}